\documentclass[11pt]{article}
\usepackage{color}
\usepackage{graphicx}
\usepackage{amsmath,amsfonts,amssymb,graphics,amsthm}
\usepackage{hyperref}
\usepackage{comment}
\usepackage{tabularx}
\usepackage[protrusion=true,expansion=true]{microtype}
\usepackage{enumerate}
\usepackage{bbm}
\usepackage{mathrsfs}
\usepackage[margin=1.2in]{geometry}

\hypersetup{
    colorlinks=false,
    linktocpage,
    }

\numberwithin{equation}{section}

\newtheorem{theorem}{Theorem}
\newtheorem{thm}[theorem]{Theorem}
\newtheorem{corollary}[theorem]{Corollary}
\newtheorem{lemma}[theorem]{Lemma}

\newtheorem{prop}[theorem]{Proposition}

\newtheorem{cor}[theorem]{Corollary}
\newtheorem{remark}[theorem]{Remark}

\theoremstyle{remark}

\def\@rst #1 #2other{#1}
\newcommand\MR[1]{\relax\ifhmode\unskip\spacefactor3000 \space\fi
  \MRhref{\expandafter\@rst #1 other}{#1}}
\newcommand{\MRhref}[2]{\href{http://www.ams.org/mathscinet-getitem?mr=#1}{MR#2}}

\def\MR#1{\href{http://www.ams.org/mathscinet-getitem?mr=#1}{MR#1}}

\newcommand{\C}{\mathbbm{C}}

\newcommand{\E}{\mathbbm{E}}
\newcommand{\N}{\mathbbm{N}}

\newcommand{\R}{\mathbbm{R}}
\renewcommand{\P}{\mathbbm{P}}

\newcommand{\eps}{\varepsilon}
\newcommand{\1}{\mathbf{1}}

\def\cY{\mathcal{Y}}
\def\cX{\mathcal{X}}

\def\cL{\mathcal{L}}

\def\cB{\mathcal{B}}
\def\cA{\mathcal{A}}

\newcommand{\aryb}{\begin{eqnarray*}}
\newcommand{\arye}{\end{eqnarray*}}
\def\alb#1\ale{\begin{align*}#1\end{align*}}
\newcommand{\eqb}{\begin{equation}}
\newcommand{\eqe}{\end{equation}}
\newcommand{\eqbn}{\begin{equation*}}
\newcommand{\eqen}{\end{equation*}}

\newcommand{\rln}{\sqrt{\log n}}
\newcommand{\BB}{\mathbbm}

\newcommand{\frk}{\mathfrak}

\newcommand{\ep}{\epsilon}

\newcommand{\wt}{\widetilde}
\newcommand{\wh}{\widehat}
\newcommand{\mcl}{\mathcal}

\DeclareMathAlphabet{\mathpzc}{OT1}{pzc}{m}{it}

 \numberwithin{dummy}{section}

\def\({\left(}
\def\){\right)}

\def\R{\mathbf{R}}

\def\C{\mathbf{C}}
\def\S{\mathbf{S}}
\def\E{\mathbf{E}}
\def\P{\mathbf{P}}
\def\1{\mathbf{1}}

\DeclareMathOperator{\trace}{Tr}
\DeclareMathOperator{\vecdiv}{div}
\newcommand{\pd}[2]{\frac{\partial #1}{\partial #2}}
\renewcommand{\bar}[1]{\overline{#1}}

\newcommand{\ceil}[1]{\left\lceil #1 \right\rceil}
\def\L{\mathcal{L}}

\begin{document}

\author{
\begin{tabular}{c}Nina Holden\end{tabular}\;
\begin{tabular}{c}Yuval Peres\end{tabular}\;
\begin{tabular}{c}Alex Zhai\end{tabular}}

\title{Gravitational allocation for uniform points on the sphere}
\date{}

\maketitle

\begin{abstract}
Given a collection $\mcl L$ of $n$ points on a sphere $\S^2_n$ of
surface area $n$, a fair allocation is a partition of the sphere into
$n$ cells each of area $1$, and each associated with a distinct point
of $\mcl L$. We show that if the $n$ points are chosen uniformly at
random and the partition is defined by considering a ``gravitational''
potential defined by the $n$ points, then the expected distance
between a point on the sphere and the associated point of $\mcl L$ is
$O(\sqrt{\log n})$. We use our result to define a matching between two
collections of $n$ independent and uniform points on the sphere and
prove that the expected distance between a pair of matched points is
$O(\sqrt{\log n})$, which is optimal by a result of Ajtai, Koml\'os,
and Tusn\'ady. Furthermore, we prove that the expected number of
maxima for the gravitational potential is $\Theta(n/\log n)$. We also
study gravitational allocation on the sphere to the zero set $\cL$ of
a particular Gaussian polynomial, and we quantify the repulsion
between the points of $\cL$ by proving that the expected distance
between a point on the sphere and the associated point of $\cL$ is
$O(1)$.
\end{abstract}

\begin{figure}[h]
	\begin{center}
		\includegraphics[scale=0.78]{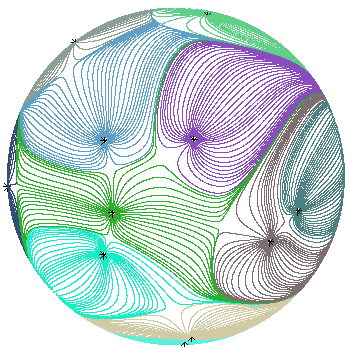}\qquad
		\includegraphics[scale=0.78]{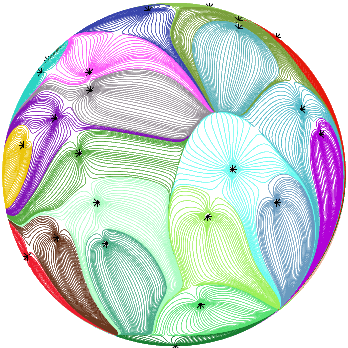}
	\end{center}
	\caption{Gravitational allocation to $n$ uniform and independent points on a sphere with $n=15$ and $n=40$. The basin of attraction of each point has equal area.}
	\label{fig1}
\end{figure}

\newpage

\tableofcontents

\section{Introduction}
Let $n$ be a positive integer, and let $\S^2_n \subset \R^3$ be the sphere centered at the origin with radius chosen such that with $\lambda_n$ denoting surface area we have $\lambda_n(\S^2_n)=n$. For any set $\L \subset \S^2_n$ consisting of $n$ points, we say that a measurable function $\psi: \S^2_n \to \L \cup \{\infty \}$ is a \emph{fair allocation} of $\lambda_n$ to $\cL$ if it satisfies the following:
\begin{equation}
\lambda_n(\psi^{-1}(\infty)) = 0, \qquad\qquad
\lambda_n(\psi^{-1}(z)) = 1, \quad\forall z\in \L.
\end{equation}
For $z\in\L$ we call $\psi^{-1}(z)$ the \emph{cell} allocated to
$z$. In other words, a fair allocation is a way to divide $\S^2_n$
into $n$ cells of measure 1 (up to a set of measure 0), with each cell
associated to a distinct point of $\L$.

Let $\L$ be a random collection of $n$ points on $\S^2_n$ which is
invariant in law under rotations of the sphere, i.e., $\phi(\L)$ has
the same law as $\L$ for any rotation $\phi:\S^2_n\to\S^2_n$. An
\emph{allocation rule} is a measurable map $\L \mapsto \psi_{\L}$
which is defined almost surely with respect to the randomness of $\L$,
such that (i) $\psi_{\L}$ is a fair allocation of $\lambda_n$ to $\L$,
and (ii) the map $\L\mapsto\psi_{\L}$ is rotation-equivariant. The
latter property means that for any rotation $\phi$ and almost every $x
\in \S^2_n$, we have $\psi_{\phi(\L)}(\phi(x))=\phi(\psi_{\L}(x))$.

\emph{Gravitational allocation} is a particular allocation rule based
on treating points in $\L$ as wells of a potential function. The cell
allocated to $z \in \L$ is then taken to be the basin of attraction of
$z$ with respect to the flow induced by the negative gradient of this
potential. When the potential takes a particular form which mimics the
gravitational potential of Newtonian mechanics, it is ensured that
a.s.\ each cell has area $1$. In this paper we will mainly consider
gravitational allocation on the sphere for the case when $\L$ is a set
of $n$ points chosen uniformly and independently at random from
$\S^2_n$.

Let us now define gravitational allocation precisely. Consider the
potential $U: \S_n^2\to\R$ given by
\begin{equation}
U(x) = \sum_{z\in\L}\log|x-z|,
\end{equation}
where $|\cdot|$ denotes Euclidean distance in $\R^3$. For each
location $x \in \S^2_n$, let $F(x)$ denote the negative gradient of
$U$ with respect to the usual spherical metric (i.e., the one induced
from $\R^3$). Note that $F(x)$ is an element of the tangent space at
$x \in \S^2_n$, and we think of it as describing the ``force'' on $x$
arising from the potential $U$.

For any $x\in\S^2_n$ consider the integral curve $Y_x(t)$ defined by
\begin{equation}
\frac{dY_x}{dt}(t) = F(Y_x(t)),\qquad Y_x(0)=x.
\end{equation}
Since $F$ is smooth away from $\L$, by standard results about flows on
vector fields (see e.g.\ the proof of Lemma 17.10 in \cite{lee}), for
each fixed $x\in\S^2_n$ the curve $Y_x$ can be defined over some
maximal domain $(-\infty, \tau_x)$, where $0 < \tau_x \le
\infty$. Note that the force $F$ represents the speed of a particle,
rather than being proportional to its acceleration as in Newtonian
gravitation.

\begin{figure}
	\begin{center}
		\includegraphics[scale=0.75]{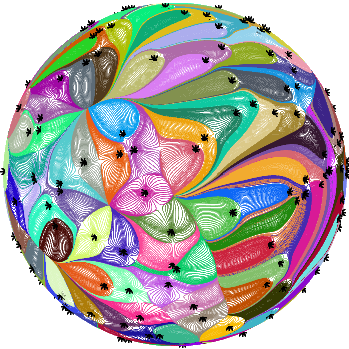}\,\,\,\,\,\,\,\,
		\includegraphics[scale=0.75]{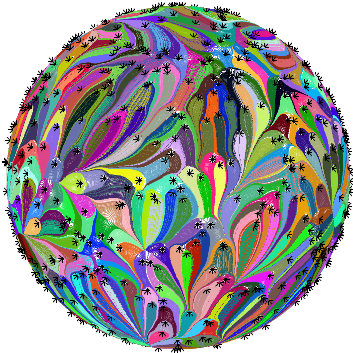}
	\end{center}
	\caption{Gravitational allocation to $n$ uniform and independent points with $n=200$ and $n=750$ (see Figure \ref{fig1} for smaller $n$). The basins become more elongated as $n$ grows, reflecting Theorem \ref{thm:finiteexpectation}. The MATLAB script used to generated these figures is based on code written by M. Krishnapur.}
	\label{fig3}
\end{figure}

\begin{figure}
	\begin{center}
		\includegraphics[scale=0.7]{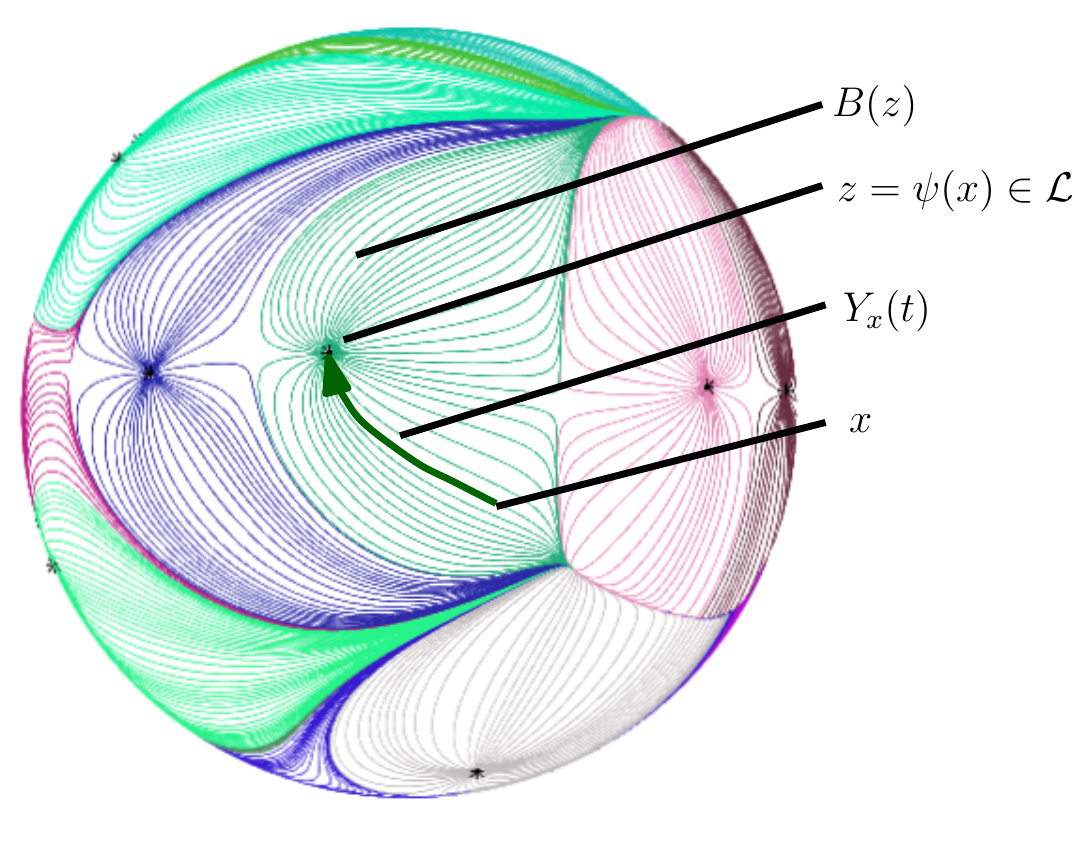}
		\includegraphics[scale=0.3]{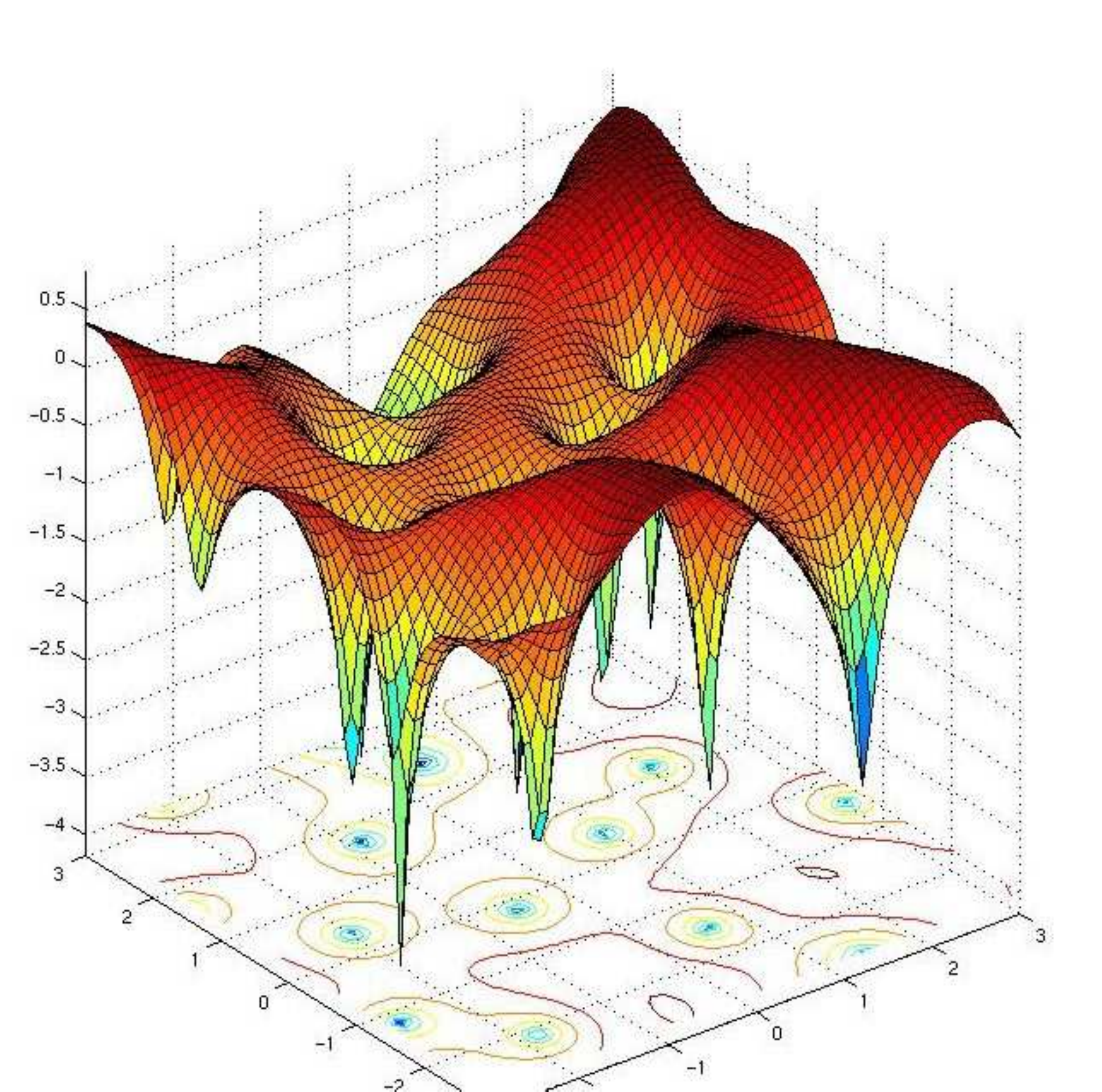}
	\end{center}
	\caption{Left: Illustration of $Y_x$, $B(z)$, and $\psi(x)$ for $x\in\S^2_n$ and $z\in\cL$. Right: Gravitational potential, by M.\ Krishnapur.}
\end{figure}

We then define \emph{gravitational allocation on the sphere} to be the
allocation rule given by
\begin{equation}
  \psi(x) = \begin{cases}
	z &\qquad\text{if\,\,} \lim_{t \uparrow \tau_x}Y_x(t)=z\text{\,\,\,and\,\,\,}z\in\L,\\
	\infty &\qquad\text{otherwise.}
  \end{cases}
  \label{eq1}
\end{equation}
For $z \in \L$, the set
\begin{equation}
  B(z) = \left\{ x\in\S^2_n\,:\,\psi(x)=z \right\}
\end{equation}
of points allocated to $z$ will be called its \emph{basin of
  attraction}.

It turns out, as stated in the following proposition, that each basin
of attraction almost surely has unit area, so that \eqref{eq1} indeed
gives rise to a fair allocation.
\begin{prop} \label{prop:allocation}
  For $n\in\N$ let $\S^2_n$ be the sphere centered at the origin with
  surface area $n$, and let $\L\subset\S^2_n$ be a set of $n$ distinct
  points. The function $\psi$ given by \eqref{eq1} defines a fair
  allocation of $\lambda_n$ to $\L$.
\end{prop}
The proof of Proposition \ref{prop:allocation} is given in Section
\ref{sec:fairallocation}. We are now ready to state the main results
of this paper.

\subsection{Statement of main results}

Our first main result estimates the average distance between a point
$x$ and the point $\psi(x)$ it is allocated to.

\begin{thm} \label{thm:finiteexpectation}
  Let $n\in\{2,3,\dots \}$. Consider any $x \in \S^2_n$, and let
  $\cL\subset \S^2_n$ be a collection of $n$ points chosen uniformly
  and independently at random from $\S^2_n$. For any $p>0$ there is a
  constant $C>0$ depending only on $p$ such that for $r>0$,
  \begin{equation} \label{eq16}
	\P\left[|\psi(x) - x|>r\sqrt{{\log n}}\right] \leq C r^{-p}.
  \end{equation}
  In particular, for some universal constant $C>0$,
  \begin{equation} \label{eq:ga-expectation-bound}
	\E[|\psi(x) - x|]\leq C \sqrt{\log n}.
  \end{equation}
\end{thm}

Gravitational allocation is optimal in the sense that
\eqref{eq:ga-expectation-bound} cannot be improved by more than a
constant factor for other allocation rules for uniform and independent
points (see Remark \ref{rmk1}).

We remark that one may obtain the bound
\eqref{eq:ga-expectation-bound} (without the tail estimate
\eqref{eq16}) more directly via the following identity, which is of
independent interest and also applied in the proof of Proposition
\ref{prop:gaf} stated below.
\begin{equation} \label{eq:L^1-identity}
\frac{1}{n} \int_{\S^2_n} \int_0^{\tau_x} |F(Y_x(t))| \,dt
\,d\lambda_n(x) = \frac{1}{2\pi n} \int_{\S^2_n} |F(x)|
\,d\lambda_n(x).
\end{equation}
Taking the expectation over $\L$, the left side upper bounds the
average value of $|\psi(x) - x|$, while the right side can be shown to
be $O(\sqrt{\log n})$ using simpler versions of estimates carried out
in Section \ref{sec:finiteexpectation}. We give the short proof of
\eqref{eq:L^1-identity} in Section \ref{sec:fairallocation}.

Fair allocations are closely related to distance-minimizing perfect
matchings between sets of points. For example, we have the following
corollary of \eqref{eq:ga-expectation-bound}. See Section
\ref{sec:matchings} for two short proofs.
\begin{cor}
	For $n\in\{2,3,\dots \}$ consider two sets of $n$ points $\cA=\{a_1,\dots,a_n\}$ and $\cB=\{b_1,\dots,b_n\}$ sampled uniformly and independently at random
	from $\S^2_n$. We can define a matching $\varphi$ of $\mcl A$ and $\mcl
	B$ (i.e., a bijection $\varphi: \mcl A \to \mcl B$) using
	gravitational allocation, such that for some universal constant $C$,
	\[ \E\left[ \frac{1}{n} \sum_{k=1}^n |\varphi(a_k)-a_k|\right]\leq C\sqrt{\log n}. \]
	\label{prop2}
\end{cor}

The next theorem shows that the expected number of local maxima of the
potential $U$ is of order $\frac{n}{\log n}$. The theorem addresses a
question of Nazarov, Sodin, and Volberg \cite[Question 12.6]{nsv07},
who, in the context of gravitational allocation to the zero set of a
Gaussian analytic function, ask about properties of the graph whose
vertices are maxima for the potential $U$ and whose edges formed by
allocation cell boundaries.
\begin{theorem}\label{thm:maxima}
  If $N \in \N \cup \{\infty\}$ denotes the number of local maxima of
  $U$, then for some universal constant $C > 0$ we have
  \[ \frac{n}{C\log n} \le \E[N] \le \frac{Cn}{\log n}. \]
\end{theorem}

As a corollary to Theorem \ref{thm:maxima} we can deduce that the
\emph{typical basin diameter} is at least of order $\sqrt{\log
  n}$.
\begin{corollary} \label{cor:diameter}
  For any $\eps>0$ there exists a $\delta>0$ such that for any fixed
  $x \in \S^2_n$, with probability at least $1 - \eps$, the cell
  containing $x$ has diameter at least $\delta\sqrt{\log n}$.
\end{corollary}
Note that \eqref{eq:ga-expectation-bound} from Theorem
\ref{thm:finiteexpectation} also gives a lower bound on $|\psi(x) -
x|$. However, the bound is only for the expectation, allowing for the
possibility that $|\psi(x) - x|$ is usually of constant order but
takes very large values with a small probability. Corollary
\ref{cor:diameter} rules out this possibility. The short proof is
deferred to Section \ref{sec:intro3}.

As mentioned above, the bound \eqref{eq:ga-expectation-bound} is
optimal among all allocation rules up to multiplication by a constant
for the case where the points of $\cL$ are uniform and
independent. However, there exist other rotationally equivariant point
processes that are spread more evenly over the sphere, and in these
cases it is possible to have $\E[|\psi(x) - x|] = O(1)$. We now
introduce one such process constructed by taking the roots of a
certain random Gaussian polynomial. Specifically, we look at the
polynomial
\begin{equation} \label{eq45}
p(z) = \sum_{k = 0}^n \zeta_k \frac{\sqrt{n(n-1) \cdots (n - k +
		1)}}{\sqrt{k!}} z^k,
\end{equation}
where $\zeta_1, \ldots , \zeta_n$ are independent standard complex
Gaussians. The roots $\lambda_1, \ldots , \lambda_n$ of $p$ are then
$n$ random points in the complex plane, which we can bring to the
sphere via stereographic projection in such a way that
\[ \L = \left\{ P_n^{-1}(\sqrt{n} \lambda_k) \right\}_{k=1}^n \]
is a rotationally equivariant random set of $n$ points on $\S^2_n$
(see Section \ref{sec:gaussian-polynomial-roots} for
details). Heuristically, the points of $\L$ are distributed more
evenly than independent uniformly random points, because roots of
random polynomials tend to ``repel'' each other (see Fig.
\ref{fig:gaussian-sim}). This can be quantified as follows.
\begin{prop}\label{prop:gaf}
  Let $\psi : \S^2_n \to \L$ be the gravitational allocation to
  $\L$. Then,
  \begin{equation} \label{eq:poly-roots-bound}
	\E\bigg[ \frac{1}{n} \int_{\S^2_n} |x - \psi(x)| d\lambda_n(x)\bigg] \leq \frac{\sqrt{\pi}}{4}.
  \end{equation}
\end{prop}

\begin{figure}
	\centering
	\includegraphics[scale=0.25]{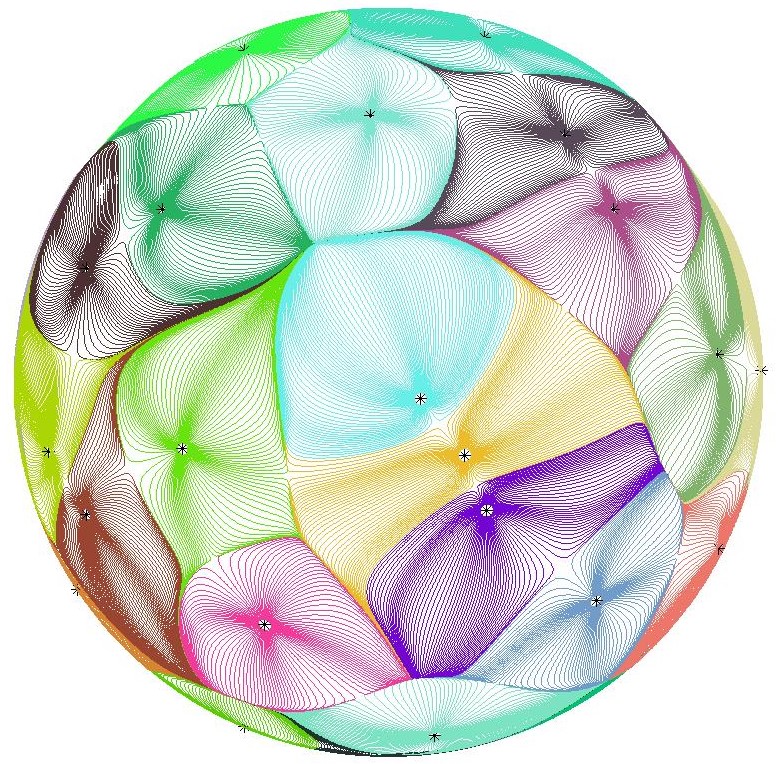}
	\caption{A simulation of gravitational allocation to the zeroes of the
		Gaussian random polynomial \eqref{eq45}. The cells are evenly proportioned, in
		contrast with the more elongated shapes seen in Figure \ref{fig3}. The simulation was made by R.\ Peled and J.\ Ding, based on code by M.\ Krishnapur.}
	\label{fig:gaussian-sim}
\end{figure}

\subsection{Related work on allocations}

Nazarov, Sodin, and Volberg \cite{nsv07} analyzed a fair allocation to
the zeros of a certain Gaussian entire function $g$, obtained from the
gradient flow determined by the potential
$U=\log|g|-\frac{1}{2}|z|^2$. The term ``gravitational allocation''
was introduced by Chatterjee, Peled, Peres, and Romik \cite{cppr10a},
who considered gravitational allocation to the points $\L \subset
\R^d$ of a unit intensity Poisson point process (PPP) for $d \geq
3$. Both papers \cite{nsv07} and \cite{cppr10a} prove an exponential
tail (with a small correction for the PPP when $d=3$) for the diameter
of the cell containing the origin. Phase transitions for the cells of
gravitational allocation to a PPP in $\R^d$ were studied in
\cite{cppr10b}.

The gravitational allocation for a PPP in $\R^d$ as studied in
\cite{cppr10a} is not well-defined for $d = 2$ because the sum
defining the force is divergent. Indeed, a lower bound for $d \le 2$
was given in \cite{hp05} (based on results from \cite{hl01,l02}): any
allocation rule for a PPP in $\R^d$ with $d = 1, 2$ satisfies
$\E[X^{d/2}]=\infty$, where $X$ is the distance between the origin and
the point it is allocated to. Nevertheless, one can study the behavior
of gravitational allocation in two dimensions by considering a finite
version of the problem, which motivates our present setting of taking
finitely many points on the sphere. Our quantitative
bounds are consistent with \cite{hl01}, because the average
distance (after appropriate scaling) grows as $\sqrt{\log n}$ with
the number of points $n$.

Gravitational allocation can also be viewed as an instantiation of the
Dacorogna-Moser \cite{dm90} scheme for a general Riemannian manifold
$D$ with volume measure $\frk m$. This scheme provides (under certain
smoothness assumptions) a coupling $\pi$ between probability measures
$\rho_0\frk m$ and $\rho_1\frk m$ by solving the PDE $\Delta
u=\rho_0-\rho_1$ and then considering the flow for the vector field
$-\nabla u$. The coupling $\pi$ is deterministic (i.e., if
$(X,Y)\sim\pi$ for $X\sim\rho_0\frk m$ and $Y\sim\rho_1\frk m$ then
$Y$ is a deterministic function of $X$), and is called a transport map
for this reason.

It was observed by Caracciolo, Lucibello, Parisi, and Sicuro
\cite{clps14} that the differential equation $\Delta u = \rho_0 -
\rho_1$ may be seen as a linearization of the Monge-Amp\`{e}re
equation, which describes the optimal transportation map for the
Wasserstein 2-distance. Based on this, they predicted the leading
order asymptotic term for optimal quadratic allocation in
$2$-dimensions (in addition to related predictions for higher
dimensions). The $2$-dimensional prediction was recently confirmed by
Ambrosio, Stra, and Trevisan \cite{ast16} for optimal quadratic
allocation cost to i.i.d.\ points sampled from a $2$-dimensional
Riemannian manifold. However, they do not obtain their result by
studying an explicit allocation method, but via a duality argument. Finer estimates with simpler proofs, for more general manifolds, and with sharper error bounds were obtained by Ambrosio and Glaudo \cite{ag18}.

Earlier works have also studied other allocation rules besides
gravitational allocation. The stable marriage allocation
\cite{hhp06,hhp09} can be defined for every translation-invariant
point process with unit intensity in $\R^d$ for $d \geq 1$: it is the
unique allocation which is stable in the sense of the Gale-Shapley
marriage problem. With this allocation, a.s.\ all cells are open and
bounded, but not necessarily connected. Allocation rules for a PPP in
$\R^d$ which minimize transportation cost per unit mass were
considered in \cite{hs13} with various cost functions, using tools
from optimal transportation.

We remark that the results of the current paper were announced in the
work \cite{pnas} by the same authors.

\subsection{Matchings: Proof of Corollary \ref{prop2} and related works}
\label{sec:matchings}

In this section we will give two short alternative proofs of Corollary \ref{prop2}, and then discuss other results on matchings.
\begin{proof}[Proof of Corollary \ref{prop2} using online matching algorithm]
	Consider the gravitational allocation $\psi$ to the point set $\mcl
	B$, and set $\varphi(a_1) = \psi(a_1)$, so that Theorem
	\ref{thm:finiteexpectation} gives
	\[ \E[|\varphi(a_1)-a_1|] \le C \sqrt{\log n}. \]
	Define
	\[ \mcl A' := \{a_2, \ldots , a_n\},\qquad
	\mcl B' := \mcl B \setminus \{ \psi(a_1) \}. \]
	Note that since $\psi$ is a fair allocation, $\psi(a_1)$ is uniformly
	distributed over elements of $\mcl B$ (under the randomness of
	$a_1$). Thus $\mcl A'$ and  $\mcl B'$ both have the law of $n-1$ points chosen independently and  uniformly at random from $\S^2_n$. Also, it is clear that  $\mcl A'$ and $\mcl B'$ are independent. Hence, we may repeat the same procedure with the sets $\mcl A'$ and $\mcl B'$ to define $\varphi(a_2)$, and we bound $|\varphi(a_2)-a_2|$ using Theorem \ref{thm:finiteexpectation} with $n-1$ points. (However, note that our  matching algorithm for $n-1$ points occurs on $\S^2_{n-1}$, so we  must rescale by a multiplicative factor $\sqrt{\frac{n}{n-1}}$.) Repeating this procedure, it follows that
	\eqbn
	\E\left[ \frac{1}{n} \sum_{k = 1}^n |\varphi(a_k)-a_k| \right]
	\le \frac{C}{n} \sum_{k = 1}^n \sqrt{\frac{n}{k}}\sqrt{\log (k\vee 2)}
	\le 2 C\log n.
	\eqen
\end{proof}

\begin{proof}[Proof of Corollary \ref{prop2} using the Birkhoff-von Neumann Theorem]
  Let $\psi_{\cA}$ and $\psi_{\cB}$ describe gravitational allocation
  to $\cA$ and $\cB$, respectively. Then, we can form a coupling
  between the uniform distributions on $\cA$ and $\cB$ as follows: we
  sample $(a, b) \in \cA \times \cB$ by drawing $X$ uniformly at
  random from $\S_n^2$ and setting $a = \psi_{\cA}(X)$ and $b =
  \psi_{\cB}(X)$.

  We have by Theorem \ref{thm:finiteexpectation} that the expected
  coupling distance satisfies the bound
  \begin{equation} \label{eq:avg-coupling-distance}
    \E|a - b| \le \E|a - X| + \E|b - X| \le 2C \sqrt{\log n}.
  \end{equation}
  By the Birkhoff-von Neumann theorem (see e.g.\ \cite[Theorem
    5.5]{lw01}), any coupling between two uniform distributions on $n$
  elements is a mixture of deterministic matchings between the two
  sample spaces. Thus, there exists some matching between $\cA$ and
  $\cB$ whose average matching distance is upper bounded by the
  quantity in \eqref{eq:avg-coupling-distance}, i.e., the average
  matching distance is of order $\sqrt{\log n}$.
\end{proof}

\begin{figure}
	\includegraphics[scale=1.15]{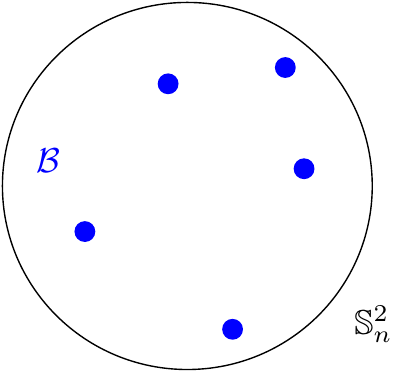}\qquad
	\includegraphics[scale=1.15]{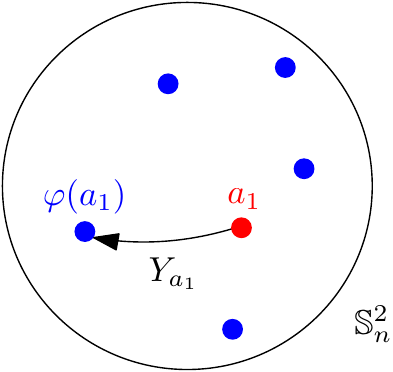}\qquad
	\includegraphics[scale=1.15]{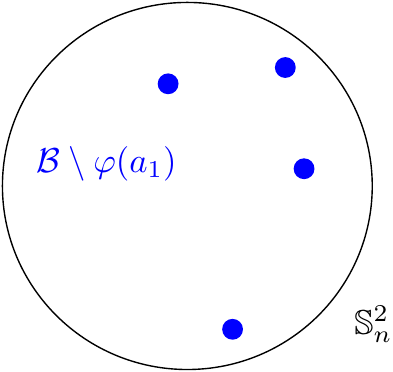}
	\caption{Illustration of the proof of Corollary \ref{prop2}. The set $\mcl B\setminus \varphi(a_1)$ consists of $n-1$ uniform and independent points on the sphere $\S_n^2$ of area $n$. }
\end{figure}
\begin{remark}
  Each proof of the corollary gives a general procedure for obtaining
  a matching from an allocation rule. In particular, we see from the
  second proof that if $\cA, \cB \subset \S^2_n$ are two sets of $n$
  points, and $\psi_{\cA}$ and $\psi_{\cB}$ are fair allocations of
  $\lambda_n$ to $\cA$ and $\cB$, respectively, then there exists a
  matching $\varphi: \cA \to \cB$ such that
  \begin{equation} \label{eq:matching-allocation-bound}
	\sum_{a \in \cA} |a - \varphi(a)| \le \int_{\S^2_n}\!\! |x - \psi_{\cA}(x)| d\lambda_n(x) + \int_{\S^2_n}\!\! |x - \psi_{\cB}(x)| d\lambda_n(x).
  \end{equation}
\end{remark}

Minimal matchings of random points in the plane have been extensively
studied (see e.g.\ \cite{akt84,ls89,t94,t14}). The asymptotic behavior
of the minimal average matching distance was identified in \cite{akt84}: it
was shown that for two sets $\mcl{A}$ and $\mcl{B}$ of $n$
i.i.d.\ uniformly chosen points from $[0,\sqrt n]^2$, there exist constants
$C_1,C_2>0$ such that
\begin{equation*}
\lim_{n \rightarrow \infty} \P\( C_1 \sqrt{\log n} \le
\frac{1}{n} \min_{\substack{\varphi: \mcl{A} \to \mcl{B} \\ \text{bijective}}} \sum_{a \in \mcl{A}} |\varphi(a)-a| \le C_2 \sqrt{\log n} \) = 1.
\end{equation*}
In the limit as $n \rightarrow \infty$, one expects minimal matching
on the sphere to be essentially equivalent to minimal matching in a
square, as the local geometries are the same to first order. Indeed,
we give a formal statement of one direction of this equivalence in the
next proposition, which is proved in Section \ref{sec:sphere-square}.
\begin{prop} \label{prop:sphere-square-matching}
	Consider any integer $n \ge 2$, and write $N = n^2$. Suppose that
	$\mcl{X}$ and $\mcl{Y}$ are two sets of $N$ i.i.d.\ uniformly random
	points from $\S^2_N$, and $\mcl{A}$ and $\mcl{B}$ are two sets of
	$n$ i.i.d.\ uniformly random points from $[0, \sqrt{n}]^2$. Then,
	for a universal constant $C$,
	\begin{equation*}
	\frac{1}{n} \E \min_{\substack{\varphi: \mcl{A} \to B \\ \text{bijective}}} \; \sum_{a \in \mcl{A}} |\varphi(a)-a| \le C + \frac{C}{N} \cdot \E \min_{\substack{\varphi: \mcl{X} \to \mcl{Y} \\ \text{bijective}}} \; \sum_{x \in \mcl{X}} |\varphi(x)-x|.
	\end{equation*}
\end{prop}
\begin{remark}\label{rmk1}
	Combined with \cite{akt84}, Proposition \ref{prop:sphere-square-matching} implies that the bound of Corollary \ref{prop2} is optimal up to multiplication by a
	constant. Using \eqref{eq:matching-allocation-bound}, we further get that gravitational allocation is optimal in the sense that the bound \eqref{eq:ga-expectation-bound} cannot be better for other allocation rules.
\end{remark}

Leighton and Shor studied the optimal \emph{maximal} matching distance for uniform points in the square. The lower bound derived in \cite{s85,s86} and the upper bound derived in \cite{ls89} show that for two sets $\mcl{A}$ and $\mcl{B}$ of $n$
i.i.d.\ uniformly chosen points from $[0,\sqrt{n}]^2$, there exist constants
$C_1,C_2>0$ such that
\begin{equation*}
\lim_{n \rightarrow \infty} \P\( C_1 (\log n)^{3/4}
\le
\min_{\substack{\varphi: \mcl{A} \to \mcl{B} \\ \text{bijective}}} \max_{a \in \mcl{A}} |\varphi(a)-a|
\le C_2 (\log n)^{3/4} \) = 1.
\end{equation*}
The maximal travel distance for the matching algorithm used in the
first proof of Corollary \ref{prop2} is of order $\sqrt{n}$, as
compared to $(\log n)^{3/4}$ for the optimal matching. However, note
that our matching algorithm is online, meaning that the points of
$\mcl A$ are revealed one by one, and we have to match a given point
of $\mcl A$ to a point of $\mcl B$ before revealing the remaining
points of $\mcl A$. The typical maximal travel distance will always be
of order $\sqrt{n}$ for online matching algorithms.

The allocation and matching problems for uniform points have also been
studied for domains of dimension $d$ not necessarily equal to $2$, and
with cost function given by the $p$-th power of the distance for $p
\geq 1$. Asymptotic results for the optimal allocation or matching
have been obtained for $d = 1$ or $2$ and all $p \geq 1$ \cite{cs14,
  akt84} as well as for $d \geq 3$ and certain $p\geq 1$
\cite{bm02,dy95,b13,fg15}.

\subsection{Proof outline for distance bound (Theorem \ref{thm:finiteexpectation}) and a heuristic picture}
\label{sec:intro3}

In order to bound $|\psi(x) - x|$ we will bound separately the duration
$\tau_x$ of the flow $Y_x$ and its speed $|F(Y_x(t))|$ for
$t\in(0,\tau_x)$. The probability distribution of $\tau_x$ may be
calculated exactly using Liouville's theorem (Proposition
\ref{prop:liouville}) and turns out to be exponential (with a constant
mean independent of $n$).

It remains to control $|F(y)|$, which turns out to be of order
$\sqrt{\log n}$. If it is always less than $\sqrt{\log n}$, then
combining with the tail bounds for $\tau_x$ yields the theorem. However, this is not precisely the case, as $|F(y)|$ can be very large
if $y$ is close to a point in $\L$. We show in Section
\ref{sec:estimates} that the contribution to $|F(y)|$ coming from
points in $\L$ outside a ball centered at $y$ of radius
$\Theta(1/\sqrt{\log n})$ is very unlikely to exceed $C\sqrt{\log n}$
for $C \gg 1$. Therefore, if $|F(y)| \gg C\sqrt{\log n}$, the main
contribution to the force is most likely coming from points of $\L$
rather close to $y$. In that case, we argue (see Lemma
\ref{lem:large-force-is-swallowed}) that one of these nearby points
typically is the point of attraction for $y$ under the gravitational flow,
which gives a bound for the distance traveled when $|F(y)|$ is
large.

The simulations in Figure \ref{fig3} suggest that the cells formed by
gravitational allocation on the sphere are long and thin. This
qualitative picture is depicted in Figure \ref{fig2}, and the
accompanying description gives a heuristic argument along the lines of
our proof outline above for why this is the case.

\begin{figure}
	\centering
	\includegraphics[scale=1]{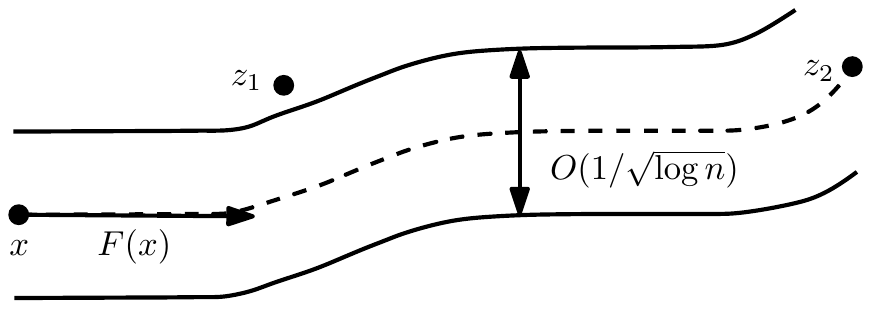}
	\caption{The figure gives a heuristic argument justifying Theorem \ref{thm:finiteexpectation}. Since the ``force'' $F$ is typically of order $\sqrt{\log n}$ and the force exerted by a point at distance $d$ has magnitude of order $1/d$, the force exerted on a particle is dominated by the forces from far away points, except if there is a point of $\cL$ at distance of order $1/\sqrt{\log n}$ or less. The probability that there are no points in a strip of width $1/\sqrt{\log n}$ and length $R\sqrt{\log n}$ decays exponentially in $R$, which suggests (heuristically) that the probability of traveling further than $R\sqrt{\log n}$ should be smaller than $\exp(-cR)$ for some $c>0$.}
	\label{fig2}
\end{figure}

Finally we provide the proof of Corollary \ref{cor:diameter} upon an application of Theorem \ref{thm:maxima}.
\begin{proof}[Proof of Corollary \ref{cor:diameter}]
	Let $E$ be the event that there are less than $Cn/\log n$ local
	maxima, where $C$ is a constant depending only on $\eps$ that is
	chosen large enough so that $\P(E) \ge 1 - \eps/2$ (this is possible
	by Markov's inequality and Theorem \ref{thm:maxima}).

	Let $R = \sqrt{\frac{(\log n)\eps}{2C\pi}}$, and note that for each
	local maximum, the spherical cap of radius $R$ centered at that
	maximum has area less than $\pi R^2 = \frac{\eps \log n}{2C}$. Thus,
	on the event $E$, the total area of points on the sphere within
	distance $R$ of a local maximum is at most
	\begin{equation} \label{eq:local-maxima-caps}
	\frac{Cn}{\log n} \cdot \frac{\eps \log n}{2C} = \frac{\eps n}{2},
	\end{equation}
	meaning that at most $\frac{\eps n}{2}$ of the gravitational
	allocation cells are fully contained within spherical caps of radius
	$R$ around local maxima.

	Next, let $E'$ denote the event that the cell containing $x$ has
	some point which is not within $R$ of any local maximum. In
	particular, note that each cell in the allocation contains one point
	in $\L$ and has at least one local maximum on its boundary, so
	whenever $E'$ holds, it means that the cell containing $x$ has
	diameter at least $R$. By \eqref{eq:local-maxima-caps} and
	rotational equivariance, we have that
	\[ \P(E') \ge \P(E' \mid E) \cdot \P(E) \ge \(1 - \frac{\eps}{2}\) \P(E) \ge 1 - \eps, \]
	which gives the desired bound with $\delta = \frac{R}{\sqrt{\log n}}
	= \sqrt{\frac{\eps}{2C\pi}}$.
\end{proof}

\subsection{Organization of the paper}

The organization of the rest of the paper is as follows. In Section
\ref{sec:fairallocation}, we prove Proposition \ref{prop:allocation}
establishing that gravitational allocation on the sphere is in fact a
fair allocation. We will then carry out most of our proofs in the
complex plane under stereographic projection rather than directly on
the sphere. Basic facts about converting between the coordinate
systems are recorded in Section \ref{sec:prelim}, which also contains
a restatement of Theorem \ref{thm:finiteexpectation} in terms of the
plane (given as Theorem \ref{thm:main-planar}). Section \ref{sec:finiteexpectation} contains the proof of Theorem
\ref{thm:main-planar} (and hence Theorem \ref{thm:finiteexpectation}),
with the proof of the main technical estimate deferred until Section
\ref{sec:estimates}. In Section \ref{sec:sphere-square} we relate
matchings on the sphere to matchings in a square (by proving
Proposition \ref{prop:sphere-square-matching}), and in Section
\ref{sec:maxima} we prove Theorem \ref{thm:maxima} on the number of
local maxima of the potential. Section
\ref{sec:gaussian-polynomial-roots} studies gravitational allocation to the zero set of the Gaussian polynomial \eqref{eq45} by proving Proposition \ref{prop:gaf}. Finally, we present a short list of open
problems in Section \ref{sec:openproblems}.

\subsection*{Acknowledgements}

We thank Manjunath Krishnapur for useful discussions as well as
sharing his code for producing simulations. We also thank Weston
Ungemach for reference suggestions related to Liouville's
theorem. Most of this work was carried out while N.\ Holden and
A.\ Zhai were visiting Microsoft Research in Redmond; they thank
Microsoft for the hospitality.

\section{Proof that gravitational allocation is a fair allocation} \label{sec:fairallocation}

In this section, we prove Proposition \ref{prop:allocation}. The
non-trivial property to verify is that for each $z \in \L$, we have
$\lambda_n(B(z)) = 1$ almost surely. Let $\Delta_S$ denote the
spherical Laplacian (i.e., the Laplace-Beltrami operator on the
sphere). The key property of our potential $U$ is that $\Delta_S U$ is
constant outside of $\L$, as seen in the next proposition.

\begin{prop} \label{prop:Delta-u}
	For a given $z \in \S^2_n$, let $g : \S^2_n \to \R$ be given by $g(x) =
	\log |x - z|$. We have
	\[ \Delta_S g(x) = 2\pi \delta_z - \frac{2\pi}{n}. \]
	Consequently,
	\[ \Delta_S U(x) = 2\pi \sum_{z \in \L} \delta_z - 2\pi. \]
	(We view $\delta_z$ as a distribution where $\int_{\S^2_n} g(x)
	\delta_z(x) \,d\lambda_n(x) = g(z)$ for any test function $g : \S^2_n \to
	\R$.)
\end{prop}
\begin{proof}
	Without loss of generality, we may assume $z = (0, 0, r_n)$, where $r_n=\sqrt{\frac{n}{4\pi}}$ is the radius of the sphere. In
	spherical coordinates, we then have $g(\theta, \phi) = \log\( 2
	\sin(\phi/2) \)+\log r_n$, where $\theta$ and $\phi$ denote the azimuthal and polar angles, respectively. Using the formula for $\Delta_S$ in spherical coordinates, we find that
	\begin{align*}
	\Delta_S g &= \frac{1}{r_n^2} \frac{1}{\sin \phi} \pd{}{\phi} \( \sin \phi \cdot
	\pd{g}{\phi} \) +
	\frac{1}{r_n^2}\frac{1}{\sin^2 \phi} \pd{^2 g}{\theta^2} \\
	&= \frac{1}{r_n^2}\frac{1}{\sin \phi} \pd{}{\phi} \cos^2(\phi/2)
	= \frac{2\pi}{n},
	\end{align*}
	which is valid at all points other than $z$.

	Since the integral of $\Delta_S g(x)$ with respect to area measure
	over $\S^2_n$ must be $0$, we deduce that $\Delta_S g = 2\pi \delta_z
	- \frac{2\pi}{n}$.
\end{proof}

Proposition \ref{prop:Delta-u} already gives an informal proof of
Proposition \ref{prop:allocation} via the divergence theorem. Consider
any $z \in \L$. If we assume that the cells $B(z)$ have piecewise
smooth boundaries, and then note that $F(x)$ is parallel to $\partial
B(z)$ at points $x\in\partial B(z)$ for which the boundary is smooth,
we get
\[ 2\pi - 2\pi \lambda_n(B(z))
= \int_{B(z)} \Delta_S U d\lambda_n
= \int_{B(z)} \vecdiv F d\lambda_n
= -\int_{\partial B(z)} F \cdot \mathbf{n}\,ds
= 0. \]

We give the formal proof using a slightly different approach
(following \cite{cppr10a}) involving Liouville's theorem for
calculating change of volume under flows, which will also be needed in
proving Theorem \ref{thm:finiteexpectation}. Conveniently, this
approach allows us to sidestep the technicalities involved in
analyzing the boundary of $B(z)$.\footnote{We also believe, however, that  the technicalities are not too hard to overcome using arguments  similar to those in \cite[Section 7]{nsv07}.} We now state the version of Liouville's theorem we need.

\begin{prop}[Liouville's Theorem] \label{prop:liouville}
	Let $M$ be an oriented $n$-dimensional Riemannian manifold, and let
	$d \alpha$ denote its volume form. Consider a smooth vector field
	$X$ on $M$.

	Let $\Phi_t$ denote the flow induced by $X$, where $\Phi_t(x) \in M$
	is defined for all $(x, t)$ in some maximal domain $\mathcal{D}
	\subseteq M \times \R$. Let $\Omega \subseteq M$ be an open set with
	compact closure. Then,
	\[ \left.\frac{d}{dt}\right|_{t=0} \int_{\Phi_t(\Omega)} \,d\alpha  = \int_{\Omega} \vecdiv X \,d\alpha. \]
\end{prop}
\begin{proof}
	Since the maximal domain $\mathcal{D}$ is open (see proof of Theorem
	17.9 in \cite{lee}) and the closure of $\Omega$ is compact, we know
	that $\Phi_t(\Omega)$ is actually defined for all $t$ in some open
	interval containing $0$. The result then follows from the formulas
	used in proving Proposition 18.18 in \cite{lee}, where the
	smoothness of the relevant $n$-forms allows us to interchange
	integration over $\Omega$ and differentiation with respect to $t$.
\end{proof}

Recall that for $x \in \S^2_n$ we wrote $(-\infty, \tau_x)$ for the
maximal domain for which $Y_x(t)$ is defined.
\begin{lemma} \label{lem:allocation-helper}
	Fix $z \in \L$, and for $t \ge 0$, define
	\[ E_t = \{x\in B(z)\,:\, \tau_x > t \}, \qquad V_t = \lambda_n(E_t). \]
    Let $\Phi_t$ denote the gravitational flow for time $t$. Then,
    $\Phi_t(E_t) = E_0$, and the pushforward of $\lambda_n$ (as a
    measure on $E_t$) under $\Phi_t$ is equal to $e^{-2\pi
      t}\lambda_n$ (as a measure on $E_0$). In particular, we have
    $V_t = e^{-2\pi t} V_0$.
\end{lemma}

\begin{figure}
	\begin{center}
		\includegraphics{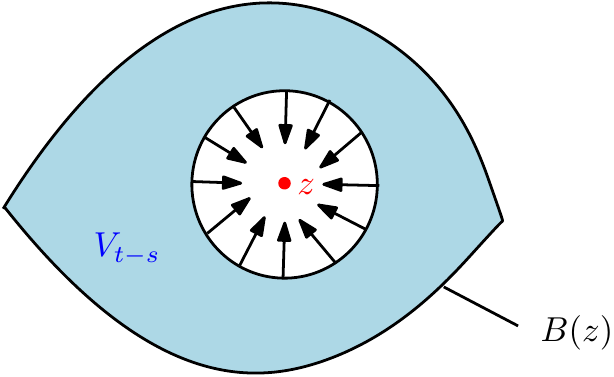}
	\end{center}
	\caption{Illustration of the proof of Lemma \ref{lem:allocation-helper}.}
\end{figure}

\begin{proof}
  We apply Proposition \ref{prop:liouville} to $\S^2_n \setminus \L$
  with the vector field $X = F = -\nabla_S U$, so that $\vecdiv X =
  \vecdiv F = -\Delta_S U$.

  Recall that $\Phi_{-s}(x) = Y_x(-s)$ is defined for all $x \in
  \S^2_n$ and $s \in (0, \infty)$. Thus, for all $s \in (0,t)$, we
  have that $\Phi_s$ is a bijection from $E_t$ to $E_{t-s}$ (with
  inverse $\Phi_{-s}$). Now, consider any $\Omega \subseteq E_t$ that
  is open with compact closure in $\S^2_n \setminus \L$. By
  Proposition \ref{prop:liouville}, we obtain for $0 \le s \le t$ that
  \[ \frac{d}{ds} \lambda_n(\Phi_s(\Omega)) = -\int_{\Phi_s(\Omega)} \Delta_S U \,d\lambda_n = \int_{\Phi_s(\Omega)} 2\pi \,d\lambda_n = 2\pi \lambda_n(\Phi_s(\Omega)). \]
  Solving the resulting differential equation yields
  \[ \lambda_n(\Omega) = e^{-2\pi t} \lambda_n(\Phi_t(\Omega)). \]
  Since any measurable subset of $E_0$ can be approximated by a set of
  the form $\Phi_t(\Omega)$, this shows that the pushforward of
  $\lambda_n$ under $\Phi_t$ is $e^{-2\pi t} \lambda_n$.
\end{proof}

We can now give the formal proof of Proposition \ref{prop:allocation}.

\begin{proof}[Proof of Proposition \ref{prop:allocation}]
	Consider any $z \in \L$, and define $E_t$ and $V_t$ as in Lemma
	\ref{lem:allocation-helper}. By Lemma \ref{lem:allocation-helper}, we
	have for all $t$ that
	\begin{equation} \label{eq:liouville-V}
	V_{0} - V_{t} = V_{0}(1-e^{-2\pi t}) = 2\pi V_{0} \cdot t+O(t^2).
	\end{equation}
	We will deduce that $V_0=1$ by estimating $V_0 - V_t$ in another way for
	small $t$.

	For any $x \in \S^2_n$, let us identify the tangent space $T_x \S^2_n$
	with a plane in $\R^3$ in the natural way\footnote{i.e., by $T_x \S^2_n
		\subset T_x \R^3 \cong \R^3$.}, so that $F(x)$ may be regarded
	as a vector in $\R^3$. By a direct calculation, we have $F(x) =
	\frac{z-x}{|z-x|^2} + O(1)$ for $x$ in a neighborhood of $z$. This
	implies
	\[ \frac{d}{dt} \left| Y_x(t) - z \right|^2 = 2\langle F(Y_x(t)), Y_x(t) - z \rangle = -2 + O\( |Y_x(t) - z| \). \]
	Write $E_{0,\eps} = E_0 \setminus E_\eps$. The above estimate implies
	for $\eps \rightarrow 0$ that
	\[ \sup_{x \in E_{0,\eps}} |x - z|^2 \le 2 \eps + o(\eps) \quad\text{and}\quad \inf_{x \not\in E_{0,\eps}} |x - z|^2 \ge 2 \eps - o(\eps). \]
	Thus, $E_{0,\eps}$ is bounded between spherical caps of radius
	$\sqrt{2\eps} \pm o(\sqrt{\eps})$, which means it has area $2\pi \eps +
	o(\eps)$. This gives
	\[ V_0 - V_{\eps} = \lambda_n(E_{0,\eps}) = 2\pi \cdot \eps + o(\eps). \]
	Comparing to \eqref{eq:liouville-V}, we conclude that $V_0 = 1$, as
	desired.
\end{proof}

Lemma \ref{lem:allocation-helper} is also the main observation needed
to explain the identity \eqref{eq:L^1-identity} relating travel
distance to average force. Essentially, it implies that the
gravitational flow linearly interpolates between the uniform measure
on $\S^2_n$ and the (discrete) uniform measure on $\L$. Consequently,
each gradient vector is ``flowed through'' by the same total mass. We
turn this into a formal proof below.

\begin{proof}[Proof of \eqref{eq:L^1-identity}]
  Take any $z \in \L$, and let $E_t$ be as in Lemma
  \ref{lem:allocation-helper}. Then, we have
  \begin{align*}
    & \int_{E_0} \int_0^{\tau_x} |F(Y_x(t))| \,dt \,d\lambda_n(x) =  \int_0^\infty \int_{E_0} |F(Y_x(t))| \cdot \1_{\tau_x > t} \,d\lambda_n(x) \,dt \\
    &\qquad=  \int_0^\infty \int_{E_t} |F(Y_x(t))| \cdot \,d\lambda_n(x) \,dt =  \int_0^\infty \int_{E_0} |F(x)| \cdot e^{-2\pi t} \,d\lambda_n(x) \,dt \\
    &\qquad=  \int_0^\infty e^{-2\pi t}\,dt \cdot \int_{E_0} |F(x)| \,dx = \frac{1}{2\pi} \int_{E_0} |F(x)|
  \,d\lambda_n(x).
  \end{align*}
  Note that $E_0 = B(z) \setminus \{z\}$, so averaging over all $z \in
  \L$, we obtain \eqref{eq:L^1-identity}.
\end{proof}

\section{Stereographic projection}
\label{sec:prelim}
Rather than work directly on the sphere, it is more convenient to work
in the plane via stereographic projection. We devote this section to
describing how to transform between the two coordinate systems, and we
give a restatement of Theorem \ref{thm:finiteexpectation} for the
plane.

Let $H = \R^2 \times \{0\} \subset \R^3$ denote the horizontal plane,
and let $z_0 = (0, 0, 1)$. The usual stereographic projection map $P:
\R^3\to\R^3$ is given by
\[ P(x) = z_0 + \frac{2(x - z_0)}{|x - z_0|^2}. \]
Let $r_n=\sqrt{\frac{n}{4\pi}}$ denote the radius of $\S^2_n$. We use
the rescaled version of $P$ defined by $P_n(x) := \sqrt{n}
P(r_n^{-1}x)$. The next proposition collects a few basic facts about
$P_n$; these can be verified by elementary calculations.

\begin{prop} \label{prop:P_n-properties}
	The map $P_n: \R^3\to\R^3$ has the following properties.
	\begin{itemize}
		\item For any $x, y \in \R^3$, we have
		\[ |P_n(x) - P_n(y)|^2 = \frac{4n r_n^2 \cdot |x - y|^2}{|x - r_nz_0|^2 \cdot |y - r_nz_0|^2}. \]
		\item $P_n$ is a conformal map from $\S^2_n \setminus \{r_n z_0\}$ to
		$H$. Its conformal scaling factor is $\frac{2\sqrt{n}r_n}{|x -
			r_nz_0|^2}$, i.e., if $g$ and $g'$ are the respective metrics on
		$\S^2_n \setminus \{r_n z_0\}$ and $H$, then
		\[ \( \frac{2\sqrt{n}r_n}{|x - r_nz_0|^2} \)^2 g_x = P_n^* g_{P_n(x)}. \]
	\end{itemize}
\end{prop}

Let $\wt{\L}=\{ P_n(y)\,:\,y\in\L \}$ be the image of $\L$ under
$P_n$. Note that the points of $\wt{\L}$ are drawn independently from a
measure $\mu_n$ on $\R^2$ that is the pushforward under $P_n$ of the
uniform probability measure on $\S^2_n$. For $x \in H \cong \R^2$, let
\[ \rho_n(x) = \sqrt{1 + \frac{|x|^2}{n}}. \]
From the conformal scaling in Proposition \ref{prop:P_n-properties},
it is straightforward to check that $\mu_n$ has density
\eqbn
\frac{d\mu_n}{dx} =
\frac{1}{\pi n \rho_n(x)^4}.
\eqen

Next, we give the planar version of our potential function. We define
for any $x, y \in \R^2$ the planar potential functions
\begin{equation} \label{eq:planar-U-def}
u(x, y) = u_y(x) = \log \frac{|x - y|}{\rho_n(x) \rho_n(y)}, \qquad
u(x) = \sum_{y \in {\L}} u(x, P_n(y))=\sum_{y \in \wt{\L}} u(x, y).
\end{equation}
By Proposition \ref{prop:P_n-properties}, we see that $u$ satisfies
for all ${x}, {y} \in \S^2_n \setminus \{r_n z_0 \}$
\[ u(P_n({x}), P_n({y})) = \log \( \frac{4n}{r_n^2} |{x} - {y}| \) = \log |{x} - {y}| + \log(16 \pi), \]
whence $u(P_n({x})) = U({x}) + n \log(16 \pi)$. We remark
that since we only care about the gradient of the potential, the
additive constant term $n \log(16 \pi)$ is not important.

We also define for $x, y \in \R^2$ the planar gradient
\begin{equation} \label{eq:planar-F-def}
f(x, y) = -\nabla u_y(x) = \frac{y - x}{|x - y|^2} + \frac{1}{n} \cdot \frac{x}{1 + \frac{|x|^2}{n}}, \qquad f(x) = \sum_{y \in \wt{\L}} f(x, y).
\end{equation}
Note however that $f(P_n(x))$ is not simply the pushforward of $F(x)$
under $P_n$ for $x\in\S_n^2$. Nevertheless, $f(P_n({x}))$ and $(DP_n)_x(F({x}))$
are scalar multiples of each other. To see this, we invoke the
following fact about conformal maps, which is routine to verify.

\begin{prop}
	Let $M_1$ and $M_2$ be two Riemannian manifolds of the same
	dimension, and let $g_1$ and $g_2$ be their respective
	metrics. Suppose we have a conformal mapping $h: M_1 \to M_2$, and
	let $c : M_1 \to \R$ denote the conformal scaling factor,
	i.e., $h^*g_2 = c^2 g_1$. Then, for any function $w \in
	C^\infty(M_1)$ and $x\in M_1$, we have
	\[ (Dh)_x(\nabla w) = c^2 \cdot \nabla(w \circ h^{-1}). \]
\end{prop}
\begin{proof}
	Consider any point $x \in M_1$, and its image $y = h(x) \in
	M_2$. Let $\langle \cdot, \cdot \rangle$ denote the natural pairing
	between vectors and $1$-forms. For any $v \in T_x$, we have
	\begin{align*}
	g_2((Dh)_xv, \nabla(w \circ h^{-1})) &= \langle (Dh)_xv, d(w \circ h^{-1}) \rangle = \langle v, dw \rangle = g_1(v, \nabla w) \\
	&= c^{-2} (h^*g_2)(v, \nabla w) = c^{-2} g_2((Dh)_xv, (Dh)_x(\nabla w)).
	\end{align*}
	Since $(Dh)_xv$ ranges over all elements of $T_y$, this implies
	\[ \nabla(w \circ h^{-1}) = c^{-2} (Dh)_x(\nabla w), \]
	which is the desired result upon multiplying both sides by $c^2$.
\end{proof}

\begin{cor} \label{cor:time-change}
	For any ${x} \in \S^2_n$, let $\wt x = P_n({x})$. Then, we have
	\[ (DP_n)_x(F({x})) = \frac{4nr^2_n}{|x - r_nz_0|^4} f(\wt x) = \pi \rho_n(\wt x)^4 f(\wt x). \]
\end{cor}

Since $f$ and $(DP_n)_x(F)$ are scalar multiples of each other, they
have the same integral curves up to reparameterization. Let us now
make explicit the change of parameterization.

\begin{prop} \label{prop:time-change}
	Consider any $\wt x\in \R^2$, and let ${x} = P_n^{-1}(\wt x)$. To lighten
	notation, let ${y}_t = Y_{{x}}(t)$. Define
	\eqbn
	\sigma(t) = \pi \int_0^t \rho_n(P_n({y}_s))^4 \,ds \qquad\text{and}\qquad \wt Y_{\wt x}(t) = P_n({y}_{\sigma^{-1}(t)}).
	\eqen
	Then, $\wt Y_{\wt x}(t)$ is an integral curve along $f$ starting at $\wt x$.
\end{prop}
\begin{proof}
	The result follows from the calculation
	\begin{align*}
	\frac{d}{dt} \wt Y_{\wt x}(t) &= (DP_n)_x(F({y}_{\sigma^{-1}(t)})) \cdot \frac{d}{dt} \sigma^{-1}(t) = (DP_n)_x(F({y}_{\sigma^{-1}(t)})) \cdot \frac{1}{\sigma'(\sigma^{-1}(t))} \\
	&= (DP_n)_x(F({y}_{\sigma^{-1}(t)})) \cdot \frac{1}{\pi} \rho_n(\wt Y_{\wt x}(t))^{-4} = f(\wt Y_{\wt x}(t)),
	\end{align*}
	where we have used Corollary \ref{cor:time-change} in the last step.
\end{proof}

Finally, we define the planar allocation function $\wt\psi: \R^2 \to \wt{\L}$
by $\wt\psi(\wt x) = (P_n \circ {\psi} \circ P_n^{-1})(\wt x)$. The cells
$\wt\psi^{-1}(z)$ for $z \in \wt{\L}$ will correspond to basins of attraction
under the flow induced by $f$. We can now reduce Theorem
\ref{thm:finiteexpectation} to an analogous statement in terms of the
plane.

\begin{thm}[Planar version of Theorem \ref{thm:finiteexpectation}] \label{thm:main-planar}
	For any $p > 0$ there is a constant $C_p > 0$ such that for $r <
	n^{1/3}$ we have
	\[ \P\left[|\wt\psi(0)| > r\sqrt{{\log n}}\right] \leq C_p r^{-p}. \]
\end{thm}

\begin{proof}[Proof of Theorem \ref{thm:finiteexpectation} from Theorem \ref{thm:main-planar}]
	By rotational symmetry, we may assume without loss of generality
	that $x = P_n^{-1}(0)$. Proposition \ref{prop:P_n-properties} ensures that we have $|x -
	{\psi}(x)| \le C |\wt\psi(0)|$ for a universal constant $C$. Observe that it is sufficient to prove the theorem for $r<Cn^{1/3}$, since $|x-\psi(x)|<\sqrt{n}$. Theorem \ref{thm:main-planar} gives
	\[ \P\( |x - {\psi}(x)| > r\sqrt{\log n} \) \le \P\( |\wt\psi(0)| > r\sqrt{\log n}/C \) \le C_p (r/C)^{-p}, \]
	which is the desired inequality upon renaming of constants.
\end{proof}

\section{Tail bound for travel distance} \label{sec:finiteexpectation}

In this section, we give the proof of Theorem \ref{thm:main-planar}
following the strategy described in the introduction. For any $\Omega
\subset \R^2$, write
\[ f(x \mid \Omega) := \sum_{y \in \wt{\L} \cap \Omega} f(x, y). \]
The following lemma, whose proof is deferred to Section
\ref{sec:estimates}, gives an upper bound of order $\sqrt{\log n}$ for
the magnitude of $f$ at points not too close to $\wt{\L}$.

\begin{lemma} \label{lem:main-disk-bound}
	There is a constant $c > 0$ such that for any $M_1, M_2, M_3 \ge 1$
	with $M_1 < n^{1/3} / \sqrt{\log n}$, and with $\delta =
	\frac{1}{M_3 \sqrt{\log n}}$, we have
	\[ \P\( \max_{x \in B(0, M_1 \sqrt{\log n})} |f(x \mid \R^2 \setminus B(x, \delta))| > M_2 \sqrt{\log n} \) \le M_1^2 e^{-cM_2/M_3 + O(1)}. \]
\end{lemma}

The next two lemmas control the behavior of points at which the
magnitude of $f$ is large.

\begin{lemma} \label{lem:star-swallowing-bound}
	Suppose $x \in \R^2$ and $r > 0$, and consider any $w \in \partial
	B(x, r)$. Let $n_w = \frac{1}{r}(w - x)$ denote the outward pointing
	unit normal vector to $\partial B(x, r)$ at $w$. Then, for any $y
	\in B(x, r)$, we have
	\[ \langle f(w, y), n_w \rangle \le -\frac{1}{2r} + \frac{|w|}{n}. \]
\end{lemma}
\begin{proof}
	Let $a = y - x$ and $b = w - x$. Note that
	\begin{align*}
	2 \langle y - w, w - x \rangle &= 2 \langle a - b, b \rangle = -2|b|^2 + 2\langle a, b \rangle \\
	&\le -|a|^2 - |b|^2 + 2\langle a, b \rangle = -|a - b|^2 \\
	&= -|y - w|^2.
	\end{align*}
	Thus,
	\begin{align*}
	\langle f(w, y), n_w \rangle &= \frac{1}{r} \left\langle \frac{y - w}{|y - w|^2} + \frac{1}{n} \cdot \frac{w}{1 + \frac{|w|^2}{n}}, w - x \right\rangle \\
	&\le \frac{\langle y - w, x - w \rangle}{r |y - w|^2} + \frac{|w|}{n}\cdot \frac{|w - x|}{r} \le -\frac{1}{2r} + \frac{|w|}{n},
	\end{align*}
	as desired.
\end{proof}

\begin{lemma} \label{lem:large-force-is-swallowed}
	Let $x \in B(0, n^{1/2}) \subset \R^2$ and $\delta \in (0, 1)$ be
	given, and define
	\[ \xi = \sup_{y \in B(x, \delta)} |f(y \mid \R^2 \setminus B(y, \delta))|. \]
	For any positive integer $k < \frac{1}{16} \sqrt{n}$, if $|f(x)| >
	((5k)^{k+1} + 1) \cdot \max(\xi, 1/\delta)$, then either $|\wt{\L} \cap
	B(x, \delta)| > k$ or $\wt\psi(x) \in B(x, 2\delta)$.
\end{lemma}
\begin{proof}
	Let $y_1, y_2, \ldots , y_m$ be the points of $\wt{\L} \cap B(x,
	\delta)$, write $\ell_i = |y_i - x|$, and assume without loss of
	generality that the $\ell_i$ are in increasing order. There is
	nothing to prove if $m > k$, so assume henceforth that $m \le k$.

	Note that since $|f(x)| > ((5k)^{k+1} + 1) \cdot
	\max(\xi, 1/\delta)$, we have by the definition of $\xi$ that
	\[ |f(x \mid B(x, \delta))| > (5k)^{k+1} \cdot \max(\xi, 1/\delta). \]
	It follows by the pigeonhole principle that $\ell_1 \le \frac{1}{5}
	\cdot (5k)^{-k} \cdot \min(\delta, 1/\xi)$. Let $j$ be the largest
	index for which $\ell_j < (5k)^j \ell_1$, and let $r = (5k)^j
	\ell_1$. Note that $r \le \delta/2$.

	Now, consider any $w \in \partial B(x, r)$, and let $n_w =
	\frac{1}{r}(w - x)$ denote the outward facing unit normal vector as
	in Lemma \ref{lem:star-swallowing-bound}. We will show that $\langle
	f(w \mid \R^2), n_w \rangle < 0$. To do this, we consider separately
	the contributions from the regions $\R^2 \setminus B(w, \delta)$,
	$B(w, \delta) \setminus B(x, r)$, and $B(x, r)$.

	For the first region, by the definition of $\xi$ (and
	recalling that $r \le \delta/2$), we have
	\begin{equation} \label{eq:swallow-outer}
	|f(w \mid \R^2 \setminus B(w, \delta))| \le \xi.
	\end{equation}
	For the second region, note that for all $i > j$, we have
	\[ \ell_i - r \ge (5k)^{j+1}\ell_1 - (5k)^j\ell_1 \ge 4kr, \]
	which implies
	\begin{align}
	|f(w \mid B(w, \delta) \setminus B(x, r))| &\le \sum_{i = j+1}^m |f(w, y_i)| \le \sum_{i = j+1}^m \( \frac{1}{\ell_i - r} + \frac{|w|}{n} \) \nonumber \\
	&\le \frac{1}{4r} + \frac{k|w|}{n}. \label{eq:swallow-middle}
	\end{align}
	Finally, for the last region we have by Lemma
	\ref{lem:star-swallowing-bound} that
	\begin{equation} \label{eq:swallow-inner}
	\langle f(w \mid B(x, r)), n_w \rangle = \left\langle \sum_{i = 1}^j f(w, x_i), n_w \right\rangle \le -\frac{j}{2r} + \frac{j \cdot |w|}{n} \le -\frac{1}{2r} + \frac{k|w|}{n}.
	\end{equation}
	Combining \eqref{eq:swallow-outer}, \eqref{eq:swallow-middle}, and
	\eqref{eq:swallow-inner}, we see that
	\begin{align*}
	\langle f(w \mid \R^2), n_w \rangle &\le \(-\frac{1}{2r} + \frac{k|w|}{n}\) + \(\frac{1}{4r} + \frac{k|w|}{n}\) + \xi \\
	&= -\frac{1}{4 \cdot (5k)^j \ell_1} + \frac{2k|w|}{n} + \xi \\
	&\le -\frac{5}{4}\max(\xi, 1/\delta) + \frac{2k|w|}{n} + \xi \\
	&\le -\frac{1}{4\delta} + \frac{2k|w|}{n} < 0. \\
	\end{align*}

	Since this holds for all $w \in \partial B(x, r)$, it follows that
	no integral curves of $f$ may escape $B(x, r)$. Consequently, we
	must have $\wt\psi(x) \in B(x, r) \subset B(x, 2\delta)$ as desired.
\end{proof}

We are now ready to prove Theorem \ref{thm:main-planar}.

\begin{proof}[Proof of Theorem \ref{thm:main-planar}]
	Note that it is enough to prove the result for sufficiently large
	$n$. We will establish the desired bound by considering the
	probabilities of three events.

	Given $p>0$ choose $k\in\{2,3,\dots \}$ and $\ep>0$ such that
	$2(k-1)-4\ep k>p$. Throughout the proof all implicit constants may
	depend on $p,k$, and $\ep$. Define $r'=r^{1-\ep}$ and
	$r''=r^{1-2\ep}$. Let $\delta = \frac{1}{r'' \sqrt{\log n}}$, and
	define the event
	\[ E_1 = \bigcap_{x\in B(0,r \rln)} \{ |\wt{\L} \cap B(x, \delta)|\leq k \}. \]
	Consider a $(\delta/2)$-net $S \subset B(0, r \sqrt{\log n})$ of
	size $O\(\frac{r^2 \log n}{\delta^2}\)$. Then,
	\begin{align}
	\P(E_1^c) &\le \sum_{s \in S} \P\( |\wt{\L} \cap B(s, 2\delta)| > k \)
	= O\(\frac{r^2 \log n}{\delta^2}\) \cdot O(\delta^{ 2(k+1) }) \nonumber \\
	&= O(r^2 \delta^{2k} \log n) = O\(\frac{r^2}{(r'')^{2k} (\log n)^{k-1}} \) \le O(r^{-p}). \label{eq:E_1-bound}
	\end{align}

	Next, let
	\[ E_2 = \left\{\max_{x \in B(0, 2r \sqrt{\log n})} |f(x \mid \R^2 \setminus B(x, \delta))| \le r' \sqrt{\log n} \right\}. \]
	According to Lemma \ref{lem:main-disk-bound}, we have
	\begin{equation} \label{eq:E_2-bound}
	\P(E_2^c) \le 4r^2 e^{-cr'/r'' + O(1)} = O\(e^{-cr^\ep/2}\) = O(r^{-p}).
	\end{equation}

	Finally, we define an event relating to the ``time traveled'' along
	integral curves of $F$. Recall the notation $Y_z(t)$ for
	the integral curve along $F$ starting at $z \in \S^2_n$. Let
	$\tau$ denote the largest time for which $Y_z(t)$ is defined
	for all $t \in (0, \tau)$; we have (almost surely) that
	$\psi(z) = Y_z(\tau)$. For $C_0:=2\pi ( (5k)^{k+1}+1 )$ define the event
	\eqbn
	E_3 = \left\{ \tau \le \frac{r^\ep}{C_0 } \right\}.
	\eqen
	According to Lemma \ref{lem:allocation-helper}, we have
	\begin{equation} \label{eq:E_3-bound}
	\P(E_3^c) \le \exp\( -2\pi \cdot \frac{r^\ep}{C_0} \) = O(r^{-p}).
	\end{equation}

	Suppose now that $E_1$, $E_2$, and $E_3$ all hold. We claim that in this
	case $|\wt\psi(0)| \le r\sqrt{\log n}$. Indeed, suppose instead that
	$|\wt\psi(0)| > r\sqrt{\log n}$.

	Let $\wt Y_0(t)$ and $\sigma$ be defined as in Proposition
	\ref{prop:time-change}, i.e., $\wt Y_0(t)$ is the integral curve along
	$f$ starting at $0 \in \R^2$, and it is related to $Y_z$ by
	\[ \wt Y_0(t) = P_n(Y_z(\sigma^{-1}(t))). \]
	Since $|\wt Y_0(0)| = 0$ and $|\wt Y_0(\sigma(\tau))| = |\wt\psi(0)|$, it then
	follows by the intermediate value theorem that there must be some
	minimal time $t' \in (0, \tau)$ for which $\wt Y_0(\sigma(t')) \in
	\partial B(0, r\sqrt{\log n})$.

	Note that from the definition of $\sigma$ in Proposition
	\ref{prop:time-change}, we have
	\[ \sigma(t')
	= \pi \int_0^{t'} \rho_n(P_n(Y_z(s)))^4 \,ds
	= \pi \int_0^{t'} \rho_n(\wt Y_0(\sigma(s)))^4 \,ds. \]
	Since $|\wt Y_0(\sigma(s))| < r\sqrt{\log n} < n^{1/3}\sqrt{\log n}$ for
	all $s < t'$, the integrand is bounded above by $2$ for sufficiently
	large $n$. Consequently, we have
	\[ \sigma(t') \le 2\pi t'. \]
	Then, by a version of the mean value theorem, we must have for some
	$s \in (0, t')$ that
	\begin{align}
	|f(\wt Y_0(\sigma(s)))| &\ge \frac{1}{\sigma(t')} |\wt Y_0(\sigma(t')) - \wt Y_0(0)| \nonumber \\
	&\ge \frac{r \sqrt{\log n}}{2\pi t'} \ge \frac{r \sqrt{\log n}}{2\pi \tau} \ge \frac{C_0}{2\pi} r' \sqrt{\log n}, \label{eq:F-mean-value-lower-bound}
	\end{align}
	where in the last step we have used the assumption that $E_3$ holds.

	Our next goal is to apply Lemma \ref{lem:large-force-is-swallowed}
	with $x = \wt Y_0(\sigma(s))$. First, we must establish that
	the hypothesis holds. Note that since $E_2$ holds, we have
	\[ \xi := \sup_{y \in B(x, \delta)} |f(y \mid \R^2 \setminus B(y, \delta))| \le r' \sqrt{\log n}. \]
	Then, \eqref{eq:F-mean-value-lower-bound} gives
	\[ |f(x)| \ge \frac{C_0}{2\pi} r' \sqrt{\log n} \ge \frac{C_0}{2\pi} \max(\xi, 1/\delta), \]
	verifying the hypothesis for Lemma \ref{lem:large-force-is-swallowed}.

	Then, we must either have $|\wt{\L} \cap B(x, 2\delta)| \le k$ or
	$\wt\psi(x) \in B(x, 2\delta)$. However, the first statement
	contradicts the assumption that $E_1$ holds, while the second
	statement contradicts $\wt\psi(x) = \wt\psi(0) \not\in B(0, r\sqrt{\log
		n})$. Thus, we conclude that whenever $E_1$, $E_2$, and $E_3$ all
	hold, then $|\wt\psi(0)| \le r\sqrt{\log n}$. In other words, we have
	using \eqref{eq:E_1-bound}, \eqref{eq:E_2-bound}, and
	\eqref{eq:E_3-bound} that
	\begin{align*}
	\P\( |\wt\psi(0)| > r\sqrt{\log n} \) \le \P(E_1^c) + \P(E_2^c) +
	\P(E_3^c) \le O(r^{-p}),
	\end{align*}
	as desired.
\end{proof}

\section{Tail bound for gravitational force}
\label{sec:estimates}
The goal of this section is to prove Lemma \ref{lem:main-disk-bound}. In fact, we will prove the closely related bound given by Lemma \ref{l:moving-disk-bound} below, from which Lemma \ref{lem:main-disk-bound} follows easily.

\begin{lemma} \label{l:moving-disk-bound}
	There is a constant $c>0$ such that for any $M \ge 1$ and $z \in
	B(0, n^{1/3})$, and with $\delta = \frac{1}{M \sqrt{\log
			n}}$, we have
	\[ \P\( \max_{x \in B(z, \sqrt{\log n})} |f(x \mid \R^2 \setminus B(x, \delta))| > t M \sqrt{\log n} \) \le e^{-ct + O(1)}. \]
\end{lemma}

\begin{proof}[Proof of Lemma \ref{lem:main-disk-bound} from Lemma \ref{l:moving-disk-bound}]
	Let $S \subset B(0, M_1 \sqrt{\log n})$ be a $\sqrt{\log n}$-net
	with $|S| = O(M_1^2)$. For each $z \in S$, we apply Lemma
	\ref{l:moving-disk-bound} to the disk of radius $\sqrt{\log n}$
	centered at $z$ with $M = M_3$ and $t = M_2/M_3$. Taking a union
	bound, we obtain
	\[ \P\( \max_{x \in B(0, M_1 \sqrt{\log n})} |f(x \mid \R^2 \setminus B(x, \delta))| > M_2 \sqrt{\log n} \) \le M_1^2 \cdot e^{-cM_2/M_3 + O(1)}, \]
	as desired.
\end{proof}

Throughout the section, we will often consider separately the effects
of points in $\wt{\L}$ within various regions. To this end, it is
convenient to extend the notation $f(x \mid \Omega)$ introduced
earlier to more general functions: for any function $H: \R^2 \times
\R^2 \to \R^k$, we write
\[ H(x \mid \Omega) := \sum_{y \in \wt{\L} \cap \Omega} H(x, y) \]
\[ \bar{H}(x \mid \Omega) := \E[ H(x \mid \Omega)] = n \int_\Omega H(x, y) \,d\mu_n(y). \]
The proof of Lemma \ref{l:moving-disk-bound}, given in Section
\ref{subsec:moving-disk-bound-proof}, uses a series of lemmas which
will occupy the remainder of this section.

\subsection{Basic estimates} \label{subsec:basic-estimates}

We first collect some basic estimates that will be used
repeatedly. Let $D_1f(x, y)$ denote the Hessian of $u_y(x)$, and let
$D_2f(x, y)$ denote the tensor of third partials of $u_y(x)$ (we may
regard $D_1f$ and $D_2f$ as elements of $\R^4$ and $\R^8$,
respectively). The following lemma follows from direct calculation
using the formula \eqref{eq:planar-F-def} for $f$.
\begin{lemma} \label{lem:basic-estimates}
	For $x \in B(0, \sqrt{n})$ and any $y$, we have the bounds
	\begin{align}
	|f(x, y)| &\le O \( \frac{1}{|x - y|} \) + O\( \frac{|x|}{n} \) \\
	|D_1f(x, y)| &\le O\( \frac{1}{|x - y|^2} \) + O\( \frac{1}{n} \) \\
	|D_2f(x, y)| &\le O\( \frac{1}{|x - y|^3} \) + O\( \frac{|x|}{n^2} \)
	\end{align}
\end{lemma}

We also give here a general exponential tail bound which will be used
repeatedly.
\begin{lemma} \label{lem:exponential-tail}
	Suppose $g:\Omega\to [-1,1]^k$ for some $\Omega \subset \R^2$. Let
	$\wt{\L}$ be a set of $n$ points drawn independently from $\mu_n$, and
	let $Y = \sum_{z \in \wt{\L} \cap \Omega} g(z)$. Then,
	\begin{equation} \label{eq:|Y|-tail}
	\log \P\( |Y| \ge t \) \le - \frac{1}{k} \cdot t + \log(2k) + 2 n \int_\Omega |g(z)| \,d\mu_n(z),
	\end{equation}
	\begin{equation} \label{eq:|Y-EY|-tail}
	\log \P\( |Y - \E Y| \ge t \) \le - \frac{1}{k} \cdot t + \log(2k) + 2 n \int_\Omega |g(z)|^2 \,d\mu_n(z).
	\end{equation}
\end{lemma}
\begin{remark}
	We will only use Lemma \ref{lem:exponential-tail} for $k \le 4$.
\end{remark}
\begin{proof}
	Write $g(x) = (g_1(x), g_2(x), \ldots , g_k(x))$, and let $x_1,
	\ldots , x_n$ be the points of $\wt{\cL}$. For $1 \le i \le k$ and $1 \le j
	\le n$, let $Z_{ij} = \1_{x_j \in \Omega} g_i(x_j)$. We use the
	inequalities
	\[ e^s \le 1 + 2s, \qquad e^s \le 1 + s + 2s^2 \]
	for $s \in [-1, 1]$. Since $|g_i(x_j)| \le 1$, we obtain
	\begin{align*}
	\E[e^{Z_{ij}}] &= 1 + \int_\Omega (e^{g_i(z)} - 1) \,d\mu_n(z) \le 1 + 2\int_\Omega |g_i(z)| \,d\mu_n(z) \\
	&\le \exp\( 2 \int_\Omega |g_i(z)| \,d\mu_n(z) \)\\
	\E[e^{Z_{ij} - \E Z_{ij}}] &= 1 + \int_\Omega (e^{g_i(z)} - g_i(z) - 1) \,d\mu_n(z) \le 1 + 2\int_\Omega |g_i(z)|^2 \,d\mu_n(z) \\
	&\le \exp\( 2 \int_\Omega |g_i(z)|^2 \,d\mu_n(z) \).
	\end{align*}
	Letting $Y_i$ denote the $i$-th coordinate of $Y$, summing the above
	bounds over all $j$ and using Markov's inequality yields
	\[ \log \P\( Y_i \ge t \) \le -t + 2n \int_\Omega |g(z)| \,d\mu_n(z) \]
	\[ \log \P\( Y_i - \E Y_i \ge t \) \le -t + 2n \int_\Omega |g(z)|^2 \,d\mu_n(z). \]

	The above inequalities also apply for $Y_i$ replaced with
	$-Y_i$. Union bounding over $1 \le i \le k$ and both choices of
	signs, and using the inequalities $|Y| \le \sum_{i = 1}^k |Y_i|$ and
	$|Y - \E Y| \le \sum_{i = 1}^k |Y_i - \E Y_i|$, we obtain
	\eqref{eq:|Y|-tail} and \eqref{eq:|Y-EY|-tail}, as desired.
\end{proof}

\subsection{Bounds of averages}
\begin{lemma} \label{lem:force-average}
	Consider a point $z \in B(0, n^{1/2})$ and a radius $R \le
	n^{1/2}$. Let $\Omega = \R^2 \setminus B(z, R)$. Then,
	\begin{align*}
	|\bar{f}(z \mid \Omega)| &= O(R).
	\end{align*}
\end{lemma}
\begin{proof}
	First, note that by rotational symmetry, we have $\E[F(z)]
	= 0$ for any $z\in\S^2_n$. Thus, by Corollary
	\ref{cor:time-change}, we have
	\begin{equation} \label{eq:plane-avg}
	\bar{f}(z \mid \R^2) = 0.
	\end{equation}
	Next, by Lemma \ref{lem:basic-estimates}, we have
	\[ |\bar{f}(z \mid B(z, R))| = \left| n \int_{B(z, R)} f(z, y) \mu_n(y) \,dy \right| \le \int_{B(z, R)} |f(z, y)| \,dy \]
	\[ = O\( \int_0^R 2\pi r \cdot \frac{1}{r} \,dr \) + O\( R^2 |z| / n \) = O(R). \]
	Combining this with \eqref{eq:plane-avg}, we obtain
	\[ |\bar{f}(z \mid \Omega)| \le |\bar{f}(z \mid \R^2)| + |\bar{f}(z \mid B(z, R))| = O(R). \]
\end{proof}

\begin{lemma} \label{lem:force-deriv-average}
	Consider a point $z \in B(0, n^{1/2})$ and a radius $R \le
	n^{1/2}$. Let $\Omega = \R^2 \setminus B(z, R)$. Then,
	\begin{align*}
	|\bar{D_1f}(z \mid \Omega)| &= O(1).
	\end{align*}
\end{lemma}
\begin{proof}
	By direct calculation, we find that
	\[ D_1f(z, y) = \frac{2 (z - y)^{\otimes 2} - |z - y|^2 I_2}{|z - y|^4} + \frac{1}{n} \( \frac{I_2}{1 + \frac{|z|^2}{n}} - \frac{2 z^{\otimes 2}}{n \(1 + \frac{|z|^2}{n}\)^2} \). \]
	Let $A(z, y)$ and $B(z, y)$ denote the first and second terms,
	respectively. Note that for any $r > 0$, we have by rotational symmetry that
	\[ \int_0^{2\pi} A(z, z + re^{i\theta}) \,d\theta = 0. \]
	Also, since $|z| \le n^{1/2}$, we have $|B(z, y)| = O(n^{-1})$ for all $y$. We then have with $h_n$ denoting the density of the measure $\mu_n$
	\begin{align}
	|\bar{D_1f}(z \mid \Omega)| &= n \left| \int_\Omega (A(z, y) + B(z, y)) \mu_n(y) \,dy \right| \nonumber \\
	&= n \left| \int_R^\infty r \int_0^{2\pi} A(z, z + re^{i\theta}) \mu_n(z + re^{i\theta}) \,d\theta \,dr \right| + O(1) \nonumber \\
	&\le n \int_R^\infty r \cdot \max_{y \in \partial B(z, r)} |A(z, y)| \cdot \max_{y, y' \in \partial B(z, r)} |h_n(y)-h_n(y')| \,dr + O(1). \label{eq:force-deriv-avg1}
	\end{align}
	To estimate the final expression, first note that $|A(x, y)| = O\(
	\frac{1}{|x - y|^2} \)$. Also, for any $r > 0$, we have
	\begin{align*}
	&\max_{y, y' \in \partial B(z, r)} |h_n(y) - h_n(y')| = \frac{1}{\pi n \( 1 + \frac{(r - |z|)^2}{n} \)^4} - \frac{1}{\pi n \( 1 + \frac{(r + |z|)^2}{n} \)^4} \\
	&\qquad \le \left| \frac{(r - |z|)^2}{n} - \frac{(r + |z|)^2}{n} \right| \cdot \frac{4}{\pi n \( 1 + \frac{(r - |z|)^2}{n} \)^5} = O\( \frac{r |z|}{n^2 \max(1, r^{10}/n^5)}\).
	\end{align*}
	Applying these estimates to \eqref{eq:force-deriv-avg1}, we have
	\begin{align*}
	|\bar{D_1f}(z \mid \Omega)| &\le O\( \int_R^\infty \frac{|z|}{n \max(1, r^{10}/n^5)} \,dr \) + O(1) \\
	&= O\(\frac{|z|}{n^{1/2}}\) + O(1) = O(1).
	\end{align*}
\end{proof}

\subsection{Far contributions}

\begin{lemma} \label{l:big-disk}
	Let $R$ be a number with $1 \le R \le n^{1/3}$. Consider any point
	$z \in B(0, n^{1/3})$, and let $\Omega = \R^2 \setminus B(z, 2R)$. Then,
	for some $c > 0$,
	\[ \P\( \max_{x \in B(z, R)} |f(x \mid \Omega)| > t (R + \sqrt{\log n}) \) \le e^{-ct + O(1)} \]
\end{lemma}
\begin{proof}
	We first claim that for small enough $c$, each of the following
	inequalities occurs with probability at least $1 - e^{-ct + O(1)}$:
	\begin{align}
	|f(z \mid \Omega)| &\le t(R + \sqrt{\log n}) \label{eq:l-big-disk-1} \\
	|D_1f(z \mid \Omega)| &\le t \label{eq:l-big-disk-2} \\
	\max_{x \in B(z, R)} |D_2f(x \mid \Omega)| &\le t/R \label{eq:l-big-disk-3}
	\end{align}
	We do this by applying Lemma \ref{lem:exponential-tail} three
	times with different functions $f$.

	First, take $g(x) = \frac{1}{C_1 \sqrt{\log n}} f(z, x)$ with $C_1$ a
	large enough constant so that Lemma \ref{lem:basic-estimates} gives
	the upper bound
	\[ |g(x)| \le \frac{1}{\sqrt{\log n}} \( \frac{1}{|z - x|} + \frac{|z|}{n} \) \le \frac{1}{\sqrt{\log n}} \( \frac{1}{|z - x|} + n^{-2/3} \). \]
	Note that this bound ensures $|g(x)| \le 1$ for all $x \in \Omega$,
	so that Lemma \ref{lem:exponential-tail} applies. Lemma
	\ref{lem:exponential-tail} then gives
	\begin{align}
	&\log \P\( |f(z \mid \Omega) - \bar{f}(z \mid \Omega)| > C_1 t \sqrt{\log n} \) \le - \frac{1}{2} t + O(1) + 2n \int_\Omega \frac{\( \frac{1}{|z - x|} + n^{-2/3} \)^2}{\log n} \,d\mu_n(x) \nonumber \\
	&\qquad\qquad \le - \frac{1}{2} t + O(1) + 4n \int_\Omega \frac{\frac{1}{|z - x|^2} + n^{-4/3}}{\log n} \,d\mu_n(x). \label{eq:l-big-disk-1a}
	\end{align}
	We estimate the integral in the last expression by observing that
	$\mu_n(x) = O(1/n)$ for all $x$, and $\mu_n(x) = O\( \frac{n}{|x -
		z|^4} \)$ for $x \not\in B(z, n^{1/2})$. Thus,
	\begin{align*}
	n \int_\Omega \frac{\frac{1}{|z - x|^2} + n^{-4/3}}{\log n} \,d\mu_n(x) &\le \frac{O(1)}{\log n} \( \int_{2R}^{\sqrt{n}} \( \frac{1}{r}+ \frac{r}{n^{4/3}} \) \,dr + \int_{\sqrt{n}}^\infty \( \frac{1}{r} + \frac{r}{n^{4/3}} \) \frac{n^2}{r^4} \,dr \) \\
	&= \frac{O(1)}{\log n} \( O(\log n) + O(1) \) = O(1).
	\end{align*}
	Substituting into \eqref{eq:l-big-disk-1a}, we obtain
	\[ \log \P\( |f(z \mid \Omega) - \bar{f}(z \mid \Omega)| > C_1 t \sqrt{\log n} \) \le - \frac{1}{2} t + O(1). \]
	By Lemma \ref{lem:force-average}, we also have $\bar{f}(z \mid
	\Omega) = O(R)$. Thus, after rescaling $t$, we see that
	\eqref{eq:l-big-disk-1} occurs with probability at least $1 - e^{-ct
		+ O(1)}$ for small enough $c$.

	Next, take $g(x) = \frac{1}{C_2} D_1f(z, x)$ with $C_2$ large enough
	so that Lemma \ref{lem:basic-estimates} gives
	\[ |g(x)| \le \frac{1}{|x - z|^2} + \frac{1}{n}. \]
	Using Lemma \ref{lem:exponential-tail}, we obtain
	\begin{align*}
	&\log \P\( |D_1f(z \mid \Omega) - \bar{D_1f}(z \mid \Omega)| > C_2 t \) \le - \frac{1}{4} t + O(1) + 2n \int_\Omega \( \frac{1}{|x - z|^2} + \frac{1}{n} \)^2 \,d\mu_n(x) \\
	&\qquad\le - \frac{1}{4} t + O(1) + 4n \int_\Omega \( \frac{1}{|x - z|^4} + \frac{1}{n^2} \) \,d\mu_n(x) \\
	&\qquad= - \frac{1}{4} t + O(1) + \frac{4}{n}\mu_n(\Omega) + O\( \int_R^\infty \frac{1}{r^4} \cdot r \,dr \) \\
	&\qquad= - \frac{1}{4} t + O(1) + O(1/n) + O(1/R^2) = - \frac{1}{4}t + O(1).
	\end{align*}
	By Lemma \ref{lem:force-deriv-average}, $|\bar{D_1f}(z \mid \Omega)|
	= O(1)$. Thus, after rescaling $t$, we see that
	\eqref{eq:l-big-disk-2} also occurs with probability at least $1 -
	e^{-ct + O(1)}$ for small enough $c$.

	Finally, we take $g(y) = \frac{R}{C_3} \cdot \max_{x \in B(z, R)} |D_2f(x, y)|$
	with $C_3$ large enough so that Lemma \ref{lem:basic-estimates} gives
	\[ |g(y)| \le \frac{R}{|y - z|^3} + Rn^{-4/3}. \]
	Using Lemma \ref{lem:exponential-tail}, we obtain
	\begin{align*}
	&\log \P\( \max_{x \in B(z, R)} |D_2f(x \mid \Omega)| > C_3 t/R \) \le \log \P\( \frac{R}{C_3} \sum_{y \in \wt{\L} \cap \Omega} \max_{x \in B(z, R)} |D_2f(x, y)| > t \) \\
	&\qquad\le -t + O(1) + 2n \int_\Omega \( \frac{R}{|y - z|^3} + Rn^{-4/3} \) \,d\mu_n(y) \\
	&\qquad\le -t + O(1) + 2n^{-1/3} R \mu_n(\Omega) + O\( \int_R^\infty \frac{R}{r^3} \cdot r \,dr \) = -t + O(1).
	\end{align*}
	Thus, after rescaling $t$, \eqref{eq:l-big-disk-3} occurs with
	probability at least $1 - e^{-ct + O(1)}$ for small enough $c$.

	Now, suppose that the inequalities \eqref{eq:l-big-disk-1},
	\eqref{eq:l-big-disk-2}, and \eqref{eq:l-big-disk-3} all hold. Then,
	\eqref{eq:l-big-disk-2} and \eqref{eq:l-big-disk-3} imply that
	\[ \max_{x \in B(z, R)} |D_1f(x \mid \Omega)| \le 2t. \]
	Combining this with \eqref{eq:l-big-disk-1} yields
	\[ \max_{x \in B(z, R)} |f(x \mid \Omega)| \le t(3R + \sqrt{\log n}), \]
	which holds with probability at least $1 - e^{-ct + O(1)}$. Rescaling
	$t$ gives the result.
\end{proof}

\subsection{Near contributions}

\begin{lemma} \label{l:near-force-bounds}
	Let $0 < q < \frac{1}{2}$ and $2 < R < n^{1/3}$ be
	given. Consider any $z \in B(0, n^{1/3})$ and any $\Omega
	\subset B(0, R) \setminus B(z, q)$. There is an absolute constant $c >
	0$ such that for all $t > 0$, we have
	\[ \P\( \max_{x \in B(z, q/2)} |f\( x \mid \Omega \) - \overline{f}(z \mid \Omega)|\ge t/q \) \le q^{ct - O(1)} e^{O(q \log R)}. \]
\end{lemma}
\begin{proof}
	Let $\Omega_1 = \Omega \setminus B(z, 1)$ and $\Omega_2 = \Omega
	\cap B(z, 1)$.

	We first apply Lemma \ref{lem:exponential-tail} twice on
	$\Omega_1$. Taking $f(y) = \frac{q}{C_1} \max_{x \in B(z, q/2)}
	|D_1f(x, y)|$ with $C_1$ large enough to ensure that $|g(y)| \le 1$
	on $\Omega_1$, we find that
	\begin{align}
	&\log \P\( \max_{x \in B(z, q/2)} |D_1f(x \mid \Omega_1)| \ge C_1 t/q \) \le -t + O(1) + 2n \int_{\Omega_1} |g(y)| \,d\mu_n(y) \nonumber \\
	&\qquad\le -t + O(1) + O(nq) \cdot \int_{\Omega_1} \( \frac{1}{|y - z|^2} + \frac{1}{n} \) \,d\mu_n(y) \\
	&\qquad\le -t + O(1) + O(q \mu_n(\Omega_1)) + O\( \int_1^{2R} \frac{q}{r^2} \cdot r \,dr \) \nonumber \\
	&\qquad= -t + O(1) + O(q \log R). \label{eq:med-D_1F-bound}
	\end{align}
	For our second application of Lemma \ref{lem:exponential-tail}, we
	take $g(y) = \frac{\sqrt{q}}{C_2} f(z, y)$ with $C_2$ large enough
	to ensure $|g(y)| \le 1$ on $\Omega_1$. We obtain
	\begin{align}
	&\log \P\( |f(z \mid \Omega_1) - \bar{f}(z \mid \Omega_1)| \ge C_2 t/\sqrt{q} \) \le -\frac{1}{2}t + O(1) + 2n \int_{\Omega_1} |g(y)|^2 \,d\mu_n(y) \nonumber \\
	&\qquad\le -\frac{1}{2}t + O(1) + O(nq) \cdot \int_{\Omega_1} \(\frac{1}{|y - z|} + \frac{|z|}{n}\)^2 \,d\mu_n(y) \nonumber \\
	&\qquad\le -\frac{1}{2}t + O(1) + O\(nq \cdot \frac{|z|^2}{n^2}\) \cdot \mu_n(\Omega_1) + O\( \int_1^{2R} \frac{q}{r^2}  \cdot r \,dr \) \nonumber \\
	&\qquad= -\frac{1}{2}t + O(1) + O(q \log R). \label{eq:med-F-bound}
	\end{align}
	Combining \eqref{eq:med-D_1F-bound} and \eqref{eq:med-F-bound} and
	rescaling $t$, we obtain
	\[ \log \P\( \max_{x \in B(z, q/2)} |f(x \mid \Omega_1) - \bar{f}(z \mid \Omega_1)| \ge t/\sqrt{q} \) \le -ct + O(1) + O(q \log R) \]
	for sufficiently small $c$. Setting $t = s/\sqrt{q}$, this may be
	rewritten as
	\begin{align}
	\P\( \max_{x \in B(z, q/2)} |f(x \mid \Omega_1) - \bar{f}(z \mid \Omega_1)| \ge s/q \) &\le e^{-cs/\sqrt{q} + O(1) + O(q \log R)} \nonumber \\
	&\le q^{cs - O(1)} \cdot e^{O(q \log R)}. \label{eq:med-bound}
	\end{align}

	Next, we analyze the contribution from $\Omega_2$. Let $g(y) =
	\frac{1}{C_3} \cdot q \log \frac{1}{q} \cdot \max_{x \in B(z, q/2)}
	|f(x, y)|$, where $C_3$ is a large enough constant so that (using
	Lemma \ref{lem:basic-estimates})
	\[ g(y) \le \frac{1}{4} \cdot q \log \frac{1}{q} \cdot \frac{1}{|y - z|} \]
	for all $y \in \Omega_2$. We cannot apply Lemma
	\ref{lem:exponential-tail} directly, because we do not have $|f(y)|
	\le 1$ on all of $\Omega_2$. However, a similar argument using a
	more precise analysis of exponential moments will work. Note that
	\begin{align*}
	\E\left[ e^{\sum_{y \in \wt{\L} \cap \Omega_2} g(y)} \right] &= \( 1 + \int_{\Omega_2} (e^{g(x)} - 1) \,d\mu_n(x) \)^n \le \exp\( n \int_{\Omega_2} (e^{g(x)} - 1) \,d\mu_n(x) \) \\
	&\le \exp\( 2 \int_q^1 \( \exp\( \frac{1}{4} \cdot q \log\frac{1}{q} \cdot \frac{1}{r} \)  - 1 \) \cdot r\,dr \) \\
	&\le \exp\( 2 \int_q^{q^{1/2}} q^{-1/4} \,dr + O\( \int_{q^{1/2}}^1 q \log \frac{1}{q} \,dr \) \) \\
	&= \exp\( O(q^{1/4}) \).
	\end{align*}
	Markov's inequality then implies
	\[ \log \P\(\max_{x \in B(z, q/2)} |f(x \mid \Omega_2)| \ge \frac{C_3 t}{q \log \frac{1}{q}} \) \le -t + O(q^{1/4}). \]
	Setting $t = \frac{1}{C_3} \log \frac{1}{q} \cdot s$, this may be rewritten as
	\[ \P\( \max_{x \in B(z, q/2)} |f(x \mid \Omega_2)| \ge s/q \) \le q^{s/C_3 - O(1)}. \]
	Note that this also implies that $|\bar{f}(z \mid \Omega_2)| =
	O(1/q)$, and so we may conclude that
	\begin{equation} \label{eq:near-bound}
	\P\( \max_{x \in B(z, q/2)} |f(x \mid \Omega_2) - \bar{f}(z \mid \Omega_2)| \ge s/q \) \le q^{cs - O(1)}
	\end{equation}
	for small enough $c$. Combining \eqref{eq:med-bound} and
	\eqref{eq:near-bound} gives the result.
\end{proof}

\subsection{Overall disk bound: Proof of Lemma \ref{l:moving-disk-bound}} \label{subsec:moving-disk-bound-proof}

\begin{proof}
	Let $\Omega = B(z, 2\sqrt{\log n})$. According to Lemma
	\ref{l:big-disk} with $R = \sqrt{\log n}$, we have for small enough
	$c$ that
	\begin{equation} \label{eq:moving-disk-bound1}
	\P\( \max_{x \in B(z, \sqrt{\log n})} |f(x \mid \R^2 \setminus \Omega)| > \frac{1}{2} t M \sqrt{\log n} \) \le e^{-ctM + O(1)}.
	\end{equation}

	We next consider contributions from within $\Omega$. Let $S \subset
	B(z, \sqrt{\log n})$ be a $\delta$-net of $B(z, \sqrt{\log n})$ with
	$|S| = O(\log n / \delta^2)$. For each $s \in S$, we apply Lemma
	\ref{l:near-force-bounds} with the region $\Omega_s := \Omega
	\setminus B(s, 2\delta)$. We use the parameters $q = 2\delta$ and $R
	= 4\sqrt{\log n}$. For a small enough $c$, this gives
	\[ \P\( \max_{y \in B(s, \delta)} |f\( y \mid \Omega_s \) - \bar{f}(s \mid \Omega_s)|\ge t/4\delta \) \le \delta^{ct - O(1)} e^{O(\delta \log \log n)} = \delta^{ct - O(1)}. \]
	Thus,
	\begin{align*}
	&\P\( \max_{y \in B(s, \delta)} |f\( y \mid \Omega \setminus B(y, \delta) \) - \bar{f}(s \mid \Omega_s)|\ge t/2\delta \) \\
	&\qquad\qquad\le \delta^{ct - O(1)} + \P\( |\wt{\L} \cap (B(s, 2\delta) \setminus B(y, \delta))| \ge t/4 \) \le \delta^{ct - O(1)}.
	\end{align*}
	Using a union bound over all $s \in S$, we obtain
	\[ \P\( \max_{\substack{s \in S \\ y \in B(s, \delta)}} |f\( y \mid \Omega \setminus B(y, \delta) \) - \bar{f}(s \mid \Omega_s)|\ge \frac{1}{2} tM \sqrt{\log n} \) \le |S| \cdot \delta^{ct - O(1)} \]
	\[ \le (\log n) \cdot \delta^{-2} \cdot \delta^{ct - O(1)} = \delta^{ct - O(1)} \le e^{-ct + O(1)}. \]
	Note that by Lemma \ref{lem:force-average} and
	\eqref{eq:moving-disk-bound1}, we have
	\[ |\bar{f}(s \mid \Omega_s)| \le |\bar{f}(s \mid \R^2 \setminus B(s, 2\delta))| + |\bar{f}(s \mid \R^2 \setminus \Omega)| \le O(1) + O(\sqrt{\log n}), \]
	so it follows that
	\[ \P\( \max_{\substack{s \in S \\ y \in B(s, \delta)}} |f\( y \mid \Omega \setminus B(y, \delta) \)| \ge \frac{1}{2} tM \sqrt{\log n} \) \le e^{-ct + O(1)}. \]
	Combining with \eqref{eq:moving-disk-bound1} completes the proof.
\end{proof}

\section{Relating matchings in squares and on spheres}
\label{sec:sphere-square}
In this section we will give the proof of Proposition \ref{prop:sphere-square-matching}.
\begin{proof}[Proof of Proposition \ref{prop:sphere-square-matching}]
	\def\dmin{d_{\text{min}}}

	Let $Q = [0, \sqrt{n\pi}]^2 \subset \R^2$, and let $\wh{Q} =
	P_N^{-1}(Q) \subset \S^2_N$. It is more convenient to consider
	$\cA$ and $\cB$ having points drawn i.i.d.\ uniformly from $Q$
	rather than $[0, \sqrt{n}]^2$; clearly, the original statement
	follows after rescaling by $\sqrt{\pi}$.

	We will construct matchings of $\cA$ to
	$\cB$ based on matchings of $\cX$ to $\cY$. We first note that $|\cX \cap
	\wh{Q}| \sim \text{Binom}(N, \lambda_N(\wh{Q}) / N)$, and since
	\[ \lambda_N(\wh{Q}) = N\mu_N(Q) = n + O(n/N), \]
	we then have
	\[ \E\left| |\cX \cap \wh{Q}| - n | \right| = O(\sqrt{n}). \]

	Moreover, conditioned on the size of $|\cX \cap \wh{Q}|$, the points
	of $\wh{\cA} := P_N(\cX \cap \wh{Q})$ are distributed i.i.d.\ on $Q$
	according to a density proportional to $\mu_N$, which is within
	$O(n/N)$ in total variation distance to uniform. It then follows by
	simple calculations that $\wh{\cA}$ may be coupled to $\cA$ so that
	\[ \E |\cA \setminus \wh{\cA}| = O(\sqrt{n}) + O\( n \cdot \frac{n}{N} \) = O(\sqrt{n}). \]
	Similarly, we may couple $\wh{\cB} := P_N(\cY \cap \wh{Q})$ to $\cB$ so
	that $\E |\cB \setminus \wh{\cB}| = O(\sqrt{n})$.

	Now, let $\wh{\varphi} : \cX \to \cY$ be a matching which minimizes
	$\sum_{x \in \cX} |x - \wh{\varphi}(x)|$, let $\dmin$ denote the
	minimal value. Define the sets
	\begin{align*}
	\cA_1 &= \cA \cap \wh{\cA} \\
	\cA_2 &= \left\{ a \in \cA_1 : (P_N \circ \wh{\varphi} \circ P_N^{-1})(a) \in \wh{\cB} \right\} \\
	\cA_3 &= \left\{ a \in \cA_2 : (P_N \circ \wh{\varphi} \circ P_N^{-1})(a) \in \cB \right\},
	\end{align*}
	which satisfy $\cA_3 \subseteq \cA_2 \subseteq \cA_1 \subseteq \cA$. We may define
	a matching $\varphi : \cA \to \cB$ by setting $\varphi(a) = (P_N \circ
	\wh{\varphi} \circ P_N^{-1})(a)$ for $a \in \cA_3$ and matching the
	remaining points in an arbitrary manner.

	Note that the distance between any two points in $Q$ is at most
	$\sqrt{2\pi n}$. Also, by rotational symmetry, we have
	\[ \E \sum_{a \in \cA_3} |a - \varphi(a)| \le 2 \E \sum_{x \in \cX \cap \wh{Q}} |x - \wh{\varphi}(x)| = \frac{2 \lambda_N(\wh{Q})}{N} \E[\dmin] = O\( \frac{n}{N}\) \cdot \E[\dmin]. \]
	Thus,
	\begin{align}
	\E \sum_{a \in \cA} |a - \varphi(a)| &\le \E \sum_{a \in \cA_3} |a - \varphi(a)| + \sqrt{2\pi n} \cdot \E |\cA \setminus \cA_3| \nonumber\\
	&= O\( \frac{n}{N}\) \cdot \E[\dmin] + \sqrt{2\pi n} \( \E |\cA \setminus \cA_1| + \E |\cA_1 \setminus \cA_2| + \E |\cA_2 \setminus \cA_3| \) \nonumber\\
	&\le O\( \frac{n}{N}\) \cdot \E[\dmin] + \sqrt{2\pi n} \( O(\sqrt{n}) + \E |\wh{\cA} \setminus \cA_2| + \E |\cB \setminus \wh{\cB}| \) \nonumber\\
	&= O\( \frac{n}{N}\) \cdot \E[\dmin] + O(n) + O(\sqrt{n}) \cdot \E |\wh{\cA} \setminus \cA_2|. \label{eq:Q-matching-bound}
	\end{align}

	It remains to estimate $\E |\wh{\cA} \setminus \cA_2|$. We will use
	the fact that for piecewise smooth curves $\gamma, \gamma' \subset
	\S^2_N$ and a rotation $\vartheta \in SO_3(\R)$ chosen uniformly at
	random, the expected number of intersections of $\gamma$ with
	$\vartheta \gamma'$ is proportional to $\frac{1}{N} \cdot |\gamma|
	\cdot |\gamma'|$. (See e.g.\ the spherical kinematic formula given in
	\cite{schneider-weil}, Theorem 6.5.6. Our statement amounts to the
	special case $j = 0$ and $A = B = \S^2$.)

	For each $x \in \cX$, let $\gamma_x$ denote the geodesic in $\S^2_N$
	connecting $x$ to $\wh{\varphi}(x)$. Then, the rotational symmetry
	of $\cX$ and the above kinematic formula give
	\begin{equation} \label{eq:A_2-bound}
	\E |\wh{\cA} \setminus \cA_2| \le \E \left| \{ x \in \cX : \gamma_x \cap \partial\wh{Q} \ne \emptyset \} \right| \le O(1) \cdot \frac{|\partial \wh{Q}|}{N} \cdot \E \sum_{x \in \cX} |\gamma_x|.
	\end{equation}
	Substituting \eqref{eq:A_2-bound} into \eqref{eq:Q-matching-bound}
	and using the facts that $|\partial \wh{Q}| = O(\sqrt{n})$ and $\E
	\sum_{x \in \cX} |\gamma_x| = O(\E[\dmin])$, we conclude that
	\[ \E \sum_{a \in \cA} |a - \varphi(a)| \le O(n) + O\( \frac{n}{N} \) \E[\dmin], \]
	which gives the desired result upon dividing by $n$.
\end{proof}

\section{Local maxima of the potential}\label{sec:maxima}

In this section we prove Theorem \ref{thm:maxima}. The upper and lower
bounds will be treated separately, but both bounds require estimates
on the probability density of $|F(x)|$ (or equivalently, on $|f(0)|$
after stereographic projection). We collect the required bounds in the
following lemma, whose proof is deferred to Section
\ref{subsec:f-density-proof}.
\begin{lemma} \label{lem:f-density}
  For any $x \in \S^2_n$ and any $\eps < 1$, consider the
  stereographic projection taking $x$ to $0$. Let $f$ be the planar
  version of $F$ as defined in \eqref{eq:planar-F-def}. Then,
  \[ \P\( |f(0) - y| < \eps \) = O\( \frac{\eps^2}{\log n} \) \]
  uniformly for all $y \in \R^2$, and
  \[ \P\( |f(0)| < \eps \;\text{and}\; \nabla f(0) \preceq -\frac{1}{5}I_2 \) = \Omega\( \frac{\eps^2}{\log n} \). \]
\end{lemma}

Throughout this section, we regard $n \in \N$ as fixed. For each $\eps
\in (0, 1)$, we will form a partition of $\S_n^2$ into a collection
$\cA(\eps)$ of $\Theta(n\eps^{-2})$ spherically convex\footnote{Recall
  that a region is \emph{spherically convex} if for any two of its
  points, the region contains a minimal geodesic between them.}
regions satisfying the following properties:
\begin{itemize}
\item The diameter of each region is at most $\eps$.
\item For each region $B \in \cA(\eps)$, there exists a point $x_B \in
  B$ such that $B$ contains all points within distance $\eps/8$ of
  $x_B$.
\end{itemize}
Furthermore, it is possible to choose these partitions so that
$\cA(\eps')$ is a refinement of $\cA(\eps)$ whenever $\eps' <
\eps$. Constructing partitions with the above properties is
straightforward; we omit the details.

\subsection{Upper bound}

For a set $B \in \cA(\eps)$, we say that $B$ is a \emph{critical set}
if it contains a local maximum for $U$ and $\sup_{x, y \in B} \|
\nabla F(x) - \nabla F(y) \|_{op} \le \frac{1}{2}$, where the tangent
spaces at $x$ and $y$ are identified by the rotation along the
spherical geodesic connecting $x$ to $y$.\footnote{The symbol $\nabla$
  when applied to functions or vector fields on the sphere refers to
  the covariant derivative. This gives us simple estimates when
  integrating over geodesics. Note however that in any case, as $\eps
  \rightarrow 0$, the local geometry approaches a flat Euclidean one
  anyway.}

\begin{proof}[Proof of Theorem \ref{thm:maxima}, upper bound]
  Suppose that $B \in \cA(\eps)$ is a critical set. Let $y \in B$ be a
  local maximum of $U$, so that $\nabla^2 U(y) \preceq 0$. Recall also
  from Proposition \ref{prop:Delta-u} that $\trace \nabla^2 U(y) =
  \Delta_S U(y) = 2 \pi$. Thus, $\| \nabla^2 U(y) \|_{op} \le 2 \pi$.

  By the definition of critical set, this means also that $\| \nabla^2
  U(x) \|_{op} \le 10$ for all $x \in B$. Consequently, integrating
  along the geodesic between $y$ and $x_B$, we have $|F(x_B)| \le 10
  \eps$. Then, by Lemma \ref{lem:f-density}, for any $B \in \cA(\eps)$
  we have
  \[ \P\( \text{$B$ is a critical set} \) \le \P\( |F(x_B)| \le 10 \eps \) \le \P\( |f(0)| \le 20 \eps \) = O\( \frac{\eps^2}{\log n} \), \]
  where we have used Corollary \ref{cor:time-change} to translate
  bounds between $F(x_B)$ and $f(0)$.

  Now, let $N(\eps)$ denote the number of critical sets in
  $\cA(\eps)$. Note that $N(\eps)$ increases as $\eps$ decreases, and
  we have $\lim_{\eps \rightarrow 0} N(\eps) = N$ almost surely over
  the randomness of $\L$ (the potential $U$ is smooth away from its
  singularities, and its local maxima are bounded away from its
  singularities). Thus, by the monotone convergence theorem, we have
  \begin{equation}
    \E[N] = \lim_{\eps \rightarrow 0} \E[N(\eps)] \le \limsup_{\eps
      \rightarrow 0} \( |\cA(\eps)| \cdot O\( \frac{\eps^2}{\log n} \) \) = O\( \frac{n}{\log n} \).
  \end{equation}
\end{proof}

\subsection{Lower bound}

Consider a set $B \in \cA(\eps)$ and the stereographic projection
sending $x_B$ to $0$. Defining $u$ and $f$ as in
\eqref{eq:planar-U-def} and \eqref{eq:planar-F-def}, we say that $B$
is a \emph{candidate set} if
\begin{itemize}
\item $\nabla^2 u(0) \preceq -\frac{1}{5}I_2$,
\item $|f(0)| \le \frac{\eps}{100}$, and
\item $|z| \ge \frac{1}{n}$ for all $z \in \L$.
\end{itemize}
We first show that for small enough $\eps$, every candidate set must
contain a local maximum. Indeed, we will show that if $B$ is a
candidate set, then $u$ has a local maximum somewhere in $B(0,
\eps/16)$.

First note that since all points in $\L$ are assumed at least distance
$\frac{1}{n}$ from the origin, by Lemma \ref{lem:basic-estimates}, we
have a uniform upper bound on $|D_2f(x)|$ for $x \in B(0, \eps/5)$
that does not depend on $\eps$. Thus, for $\eps$ small enough, we have
that $D_1f(x) = \nabla^2 u(x) \preceq -\frac{1}{6}I_2$ for all $x \in
B(0, \eps/5)$.

Now, consider $f$ as a map from $\R^2$ to $\R^2$. For any $x \in
\partial B(0, \eps/16)$, we have
\[ \langle f(x), x \rangle \le \langle f(0), x \rangle - \frac{1}{6} \cdot \frac{\eps}{16} < 0. \]
Then, we have the homotopy $f_t(x) = (1 - t)f(x) - tx$ which satisfies
$f_t(x) \ne 0$ for all $t \in [0, 1]$ and $x \in \partial B(0,
\eps/16)$. It follows by standard results about topological degree
(see e.g.\ \cite[\textsection 3]{deimling}) that $f(x) = 0$ for some $x
\in B(0, \eps/16)$, and by our earlier observation that $\nabla^2
u(x)$ is negative definite in this region, this must be a local
maximum.

Finally, by Proposition \ref{prop:P_n-properties} and taking $\eps$
small enough, the disk $B(0, \eps/16)$ in the plane corresponds to
points on the sphere with distance less than $\eps/8$ from
$x_B$. Thus, our local maximum lies within the set $B$.

\begin{proof}[Proof of Theorem \ref{thm:maxima}, lower bound]
  Let $C(\eps)$ denote the number of candidate sets in $\cA(\eps)$, so
  by the preceding discussion it suffices to lower bound
  $\E[C(\eps)]$.

  Consider any $B \in \cA(\eps)$. Lemma \ref{lem:f-density} already
  gives us that
  \begin{equation} \label{eq:candidate-prob-main-term}
    \P\( |f(0)| < \frac{\eps}{100} \;\text{and}\; \nabla^2 u(0) \preceq -\frac{1}{5}I_2 \) = \Omega\( \frac{\eps^2}{\log n} \).
  \end{equation}
  We next show that when the above occurs, very rarely does it happen
  that $|z| \le \frac{1}{n}$ for some $z \in \L$. Indeed, let $z_1,
  z_2, \ldots , z_n$ be the points in $\L$. For each $i$ and $y \in
  \R^2$, we have by Lemma \ref{lem:f-density} that
  \[ \P\( \left| y - \sum_{j \ne i} f(0, z_j) \right| < \frac{\eps}{100} \) = O\( \frac{\eps^2}{\log n} \). \]
  Applying the above with $y = f(0, z_i)$ gives the estimate
  \[ \P\( |z_i| \le \frac{1}{n} \;\text{and}\; |f(0)| < \frac{\eps}{100} \) = \P\( |z_i| \le \frac{1}{n} \) \cdot O\( \frac{\eps^2}{\log n} \) = O\( \frac{\eps^2}{n^2 \log n} \). \]
  Taking a union bound over all $z_i$, this gives
  \[ \P\( |z| \le \frac{1}{n} \;\text{for some $z \in \L$ and}\; |f(0)| < \frac{\eps}{100} \) = O\( \frac{\eps^2}{n \log n} \). \]
  Combining this with \eqref{eq:candidate-prob-main-term}, we find
  that $B$ is a candidate set with probability $\Omega\(
  \frac{\eps^2}{\log n} \)$. Thus, for all small enough $\eps$,
  \[ \E[N] \ge \E[C(\eps)] = |\cA(\eps)| \cdot \Omega\(
  \frac{\eps^2}{\log n} \) = \Omega\( \frac{n}{\log n} \), \]
  as desired.
\end{proof}

\subsection{Proof of Lemma \ref{lem:f-density}} \label{subsec:f-density-proof}

In order to prove Lemma \ref{lem:f-density}, we first analyze the
contribution $f(0, z)$ from a single point $z$ drawn from $\mu_n$. A
helpful property is that $f(0, z)$ turns out to be a mixture of
Gaussians, as explained in the following lemma.

\begin{lemma} \label{lem:gaussian-mixture}
  Let $z \in \R^2$ be a point drawn from $\mu_n$, and let $X =
  \sqrt{n} f(0, z)$. Then $X$ can be sampled as a $2$-dimensional
  Gaussian of covariance $V \cdot I_2$, where $V$ itself is a
  real-valued random variable. Moreover, the probability density
  function $p_V$ of $V$ is
  \[ p_V(x) = \frac{1}{2x^2} e^{-1/2x}. \]
\end{lemma}
\begin{proof}
  Note that
  \[ \P(|z| < t) = \int_0^t 2\pi r \cdot \frac{1}{\pi n \( 1 + \frac{r^2}{n} \)^2} \,dr = \left[ - \frac{1}{1 + \frac{r^2}{n}} \right]^t_0 = 1 - \frac{1}{1 + \frac{t^2}{n}}. \]
  Hence, since $|X| = \frac{\sqrt{n}}{|z|}$, we have
  \[ \P\(|X| > t \) = 1 - \frac{1}{1 + \frac{1}{t^2}} = \frac{1}{t^2 + 1} = \P\( |z| > \sqrt{n} \cdot t \). \]
  It follows that $X$ actually has the same distribution as
  $z/\sqrt{n}$, and so its probability density function is given by
  \[ p_X(x) = \frac{1}{\pi (1 + |x|^2)^2}. \]
  We then have the integral identity
  \[ \frac{1}{\pi (1 + r^2)^2} = \frac{1}{\pi} \int_0^\infty te^{-t} \cdot e^{-tr^2} \,dt = \int_0^\infty \frac{1}{4 \pi s^3} e^{-1/2s} \cdot e^{-r^2/2s} \,ds \]
  \[ = \int_0^\infty \frac{1}{2s^2} e^{-1/2s} \cdot \( \frac{1}{2\pi s} e^{-r^2/2s} \) \,ds, \]
  which shows that $X$ can be sampled as a 2-dimensional Gaussian of
  covariance $V \cdot I_2$, where $V$ itself is a real-valued random
  variable with density $p_V(x) = \frac{1}{2x^2} e^{-1/2x}$.
\end{proof}

The next lemma provides estimates for the sum of $n$ i.i.d.\ copies of
the random variable $V$ from Lemma \ref{lem:gaussian-mixture}, which
will be relevant when we consider the sum of the contributions to
$f(0)$ from all $n$ points.

\begin{lemma} \label{lem:sum-V}
  Let $V$ be a non-negative random variable with probability density
  $p_V(x) = \frac{1}{2x^2} e^{-1/2x}$. Let $V_1, \ldots , V_n$ be $n$
  i.i.d.\ random variables each with the same distribution as
  $V$. Then, we have
  \[ \P\( \sum_{i = 1}^n V_i^2 \le \frac{n^2}{100} \text{ and } \sum_{i = 1}^n V_i \le 4n \log n \) = \Omega(1) \]
  and
  \[ \E\(\frac{1}{\sum_{i = 1}^n V_i}\) = O\( \frac{n}{\log n} \). \]
\end{lemma}
\begin{proof}
  For the first bound, define for each $i$ the event $E_i = \{ V_i \le
  n/100 \}$, and write $E = \bigcap_{i = 1}^n E_i$. For each $i$, we
  have
  \begin{align*}
    \E[V_i \mid E_i] &= \int_0^{n/100} \frac{1}{2t} e^{-1/2t} \,dt \le \log n \\
    \E[V_i^2 \mid E_i] &= \int_0^{n/100} \frac{1}{2} e^{-1/2t} \,dt \le \frac{n}{200}.
  \end{align*}
  By the independence of the $V_i$, we thus have
  \begin{align*}
    \E\( \sum_{i = 1}^n V_i \,\middle|\, E \) \le n \log n &\implies \P\( \sum_{i = 1}^n V_i \ge 4 n \log n \,\middle|\, E \) \le \frac{1}{4} \\
    \E\( \sum_{i = 1}^n V_i^2 \,\middle|\, E \) \le \frac{n^2}{200} &\implies \P\( \sum_{i = 1}^n V_i^2 \ge \frac{n^2}{100} \,\middle|\, E \) \le \frac{1}{2}.
  \end{align*}
  Thus, we have
  \begin{align*}
    \P\( \sum_{i = 1}^n V_i^2 \le \frac{n^2}{100} \text{ and } \sum_{i = 1}^n V_i \le 4n \log n \) &\ge \frac{1}{4} \P(E) = \frac{1}{4} \P(E_1)^n \\
    &= \frac{1}{4} \(e^{-50/n}\)^n = \Omega(1).
  \end{align*}

  For the second bound, let $S = \frac{1}{n} \sum_{i = 1}^n
  V_i$. Consider the three events
  \[ E_1 = \left\{ S < \frac{1}{n} \right\}, \quad E_2 = \left\{ \frac{1}{n} \le S \le \frac{1}{48} \log n \right\}, \quad E_2 = \left\{ \frac{1}{48} \log n < S \right\}. \]
  For the first event, we have
  \begin{align}
    \E\( \1_{E_1} \cdot \frac{1}{S} \) &= \int_n^\infty \P(S < 1/s) \,ds \le \int_n^\infty \P(V_1 < 1/s) \,ds \nonumber \\
    &= \int_n^\infty e^{-s/2} \,ds = 2e^{-n}. \label{eq:E1-bound}
  \end{align}

  To control the second event, let $m = \ceil{\frac{1}{3} \log_2 n}$,
  and for each positive integer $k \le m$, let $N_k$ denote the number
  of $V_i$ with $V_i \in [2^{k-1}, 2^{k}]$. Note that
  \[ \E N_k = n \cdot \P(2^{k-1} \le V_i \le 2^{k}) = n \cdot \(e^{-2^{-k}} - e^{-2^{-k-1}}\) \ge \frac{n}{2^{k+2}}. \]
  By Hoeffding's inequality, we then have
  \[ \P\( N_k \le \frac{1}{8} \cdot n \cdot 2^{-k} \) \le \exp\( - \frac{1}{32} \cdot n \cdot 2^{-2k} \). \]

  Define the event $E' = \bigcap_{k = 1}^m \{ N_k \ge \frac{1}{8} \cdot
  n \cdot 2^{-k} \}$, and note that on the event $E'$, we have
  \[ S = \frac{1}{n} \sum_{i = 1}^n V_i = \frac{1}{n} \sum_{k = 1}^m N_k \cdot 2^{k-1} \ge \frac{1}{n} \sum_{k = 1}^m \frac{1}{8} \cdot n \cdot 2^{-k} \cdot 2^{k-1} = \frac{1}{16} m, \]
  so that $E' \cap E_2 = \emptyset$. Consequently,
  \[ \P(E_2) \le 1 - \P(E') \le \sum_{k = 1}^m \exp\( - \frac{1}{8} \cdot n \cdot 2^{-2k} \) = O(n^{-2}), \]
  and so
  \begin{equation} \label{eq:E2-bound}
    \E\(\1_{E_2} \cdot \frac{1}{S} \) \le n \cdot \P(E_2) = O(n^{-1}).
  \end{equation}
  Finally, we also have
  \[ \E\(\1_{E_3} \cdot \frac{1}{S} \) \le \frac{48}{\log n}. \]
  Combining this with \eqref{eq:E1-bound} and \eqref{eq:E2-bound}, we
  conclude that $\E\( \frac{1}{S} \) = O\( \frac{1}{\log n} \)$, as
  desired.
\end{proof}

Finally, we need an elementary estimate for certain conditional
Gaussian covariances.

\begin{lemma} \label{lem:conditional-gaussian-covariance}
  Consider an $n$-dimensional Gaussian $Z = (Z_1, Z_2, \ldots , Z_n)$,
  and write $S = \sum_{i = 1}^n Z_i$. Let $\Sigma$ and $\Sigma'$ be
  the covariance matrices of $Z$ and $Z$ conditioned on $S = 0$,
  respectively, i.e., we have
  \[ \Sigma_{ij} = \E[Z_iZ_j], \quad\text{and}\quad \Sigma'_{ij} = \E[Z_iZ_j \mid S = 0]. \]
  Then,
  \[ \sum_{i,j = 1}^n (\Sigma'_{ij})^2 \le \sum_{i,j = 1}^n \Sigma_{ij}^2. \]
\end{lemma}
\begin{proof}
  Fix any $v \in \R^n$. We have
  \begin{align*}
    \langle v, \Sigma v \rangle &= \E\( \langle v, Z \rangle^2 \) = \E\( \E\( \langle v, Z \rangle^2 \mid S \) \) \\
    &\ge \E\( \E\( \langle v, Z - \E[Z \mid S] \rangle^2 \mid S \) \) \\
    &= \E\( \langle v, Z \rangle^2 \mid S = 0 \) = \langle v, \Sigma' v \rangle.
  \end{align*}
  Since this holds for all $v$, it follows that $\Sigma' \preceq
  \Sigma$. Consequently, the Hilbert-Schmidt norm of $\Sigma'$ is less
  than or equal to that of $\Sigma$, which is the desired inequality.
\end{proof}

We are now ready to prove Lemma \ref{lem:f-density}.

\begin{proof}[Proof of Lemma \ref{lem:f-density}]
  Let $V_i$ be as in Lemma \ref{lem:sum-V}, and for each $i$, let
  $X_i$ be drawn from a Gaussian of covariance $V_i \cdot I_2$. In
  light of Lemma \ref{lem:gaussian-mixture}, we may create a coupling
  in which
  \[ f(0) = \frac{1}{\sqrt{n}} \sum_{i = 1}^n X_i. \]
  Thus, $f(0)$ is distributed as a mixture of centered Gaussians,
  where the covariance has the distribution of $\frac{1}{n} \sum_{i =
    1}^n V_i$. Let $p_f$ denote the probability density of
  $f(0)$. Then, we have by the continuity of $p_f$ and Lemma
  \ref{lem:sum-V} that
  \[ \limsup_{\eps \rightarrow 0} \eps^{-2} \P\( |f(0)| < \eps \) = p_f(0) = \E\( \frac{n}{2 \pi \sum_{i=1}^n V_i} \) = O\( \frac{1}{\log n} \), \]
  proving the first bound in the case $y = 0$. The general case
  follows similarly, since $p_f$ is maximized at $0$ (being a mixture
  of centered Gaussian densities).

  For the second bound, consider any point $z = (z_1, z_2) \in
  \R^2$. A direct calculation shows that
  \begin{align*}
    f_z(0) &= \( -\frac{z_1}{z_1^2 + z_2^2}, - \frac{z_2}{z_1^2 + z_2^2} \) \\
    \nabla f_z(0) &= \frac{1}{(z_1^2 + z_2^2)^2} \left[\begin{tabular}{cc}
      $z_1^2 - z_2^2$ & $2z_1z_2$ \\
      $2z_1z_2$ & $z_2^2 - z_1^2$
      \end{tabular}\right] - \frac{1}{n} I_2.
  \end{align*}
  Thus, writing $X_i = (x_{i,1}, x_{i,2})$ and summing over all $i$,
  we see that
  \[ \nabla f(0) = \left[\begin{tabular}{cc}
      $A$ & $B$ \\
      $B$ & $-A$
      \end{tabular}\right] - I_2, \]
  where $A = \frac{1}{n} \sum_{i = 1}^n x_{i,1}^2 - x_{i,2}^2$ and $B = \frac{2}{n} \sum_{i =
    1}^n x_{i,1}x_{i,2}$.

  Now, define the event $E = \left\{ \sum_{i = 1}^n V_i^2 \le
  \frac{n^2}{100} \text{ and } \sum_{i = 1}^n V_i \le 4n \log
  n \right\}$, so that Lemma \ref{lem:sum-V} gives $\P(E) =
  \Omega(1)$. Then,
  \begin{align*}
    \E[B^2 \mid f(0) = 0, E] &= \frac{4}{n^2} \sum_{i = 1}^n \sum_{j = 1}^n \E[x_{i,1}x_{j,1} x_{i,2}x_{j,2} \mid f(0) = 0, E] \\
    &= \frac{4}{n^2} \sum_{i, j = 1}^n \E[x_{i,1}x_{j,1} \mid f(0) = 0, E]^2 \\
    &\le \frac{4}{n^2} \sum_{i, j = 1}^n \E[x_{i,1}x_{j,1} \mid E]^2 = \frac{4}{n^2} \E\left[ \sum_{i = 1}^n V_i^2 \,\middle|\, E \right] \le \frac{1}{25},
  \end{align*}
  where the first inequality step follows from Lemma
  \ref{lem:conditional-gaussian-covariance}. By Markov's inequality
  this implies that
  \begin{equation} \label{eq:B-bound}
    \P\( |B| \ge \frac{2}{5} \,\middle|\, f(0) = 0, E \) \le \frac{1}{4}.
  \end{equation}
  A nearly identical argument shows that the above inequality also
  holds with $B$ replaced by $A$. Indeed, the quantities under
  consideration are invariant under the rotation $(x_1, x_2) \mapsto
  \( \frac{x_1 + x_2}{\sqrt{2}}, \frac{x_1 - x_2}{\sqrt{2}} \)$ which
  takes $B$ to $A$. Define the function
  \[ r(x) = \P\( \nabla f(0) \preceq - \frac{1}{5} I_2 \,\middle|\, f(0) = x, E \). \]
  Then, \eqref{eq:B-bound} and the corresponding inequality for $A$
  imply that
  \[ r(0) \ge \P\( \max(|B|, |A|) \le \frac{2}{5} \,\middle|\, f(0) = 0, E \) \ge \frac{1}{2}. \]
  Also, let $p_{f \mid E}$ denote the probability density of $f$
  conditioned on the event $E$. Note that
  \[ p_{f \mid E}(0) = \E\( \frac{n}{2\pi \sum_{i = 1}^n V_i} \,\middle|\, E \) \ge \frac{1}{8\pi \log n}. \]
  Moreover, it can be checked that $r$ and $p_{f \mid E}$ are both
  continuous functions. Thus,
  \begin{align*}
    &\liminf_{\eps \rightarrow 0} \( \eps^{-2} \cdot \P\( \nabla f(0) \preceq -\frac{1}{5} I_2 \text{ and }  |f(0)| \le \eps \) \) \\
    &\qquad\ge \liminf_{\eps \rightarrow 0} \( \eps^{-2} \cdot \P(E) \cdot \int_{x \in B(0, \eps)} p_{f \mid E}(x) r(x) \,dx \) = \Omega(1).
  \end{align*}
\end{proof}

\section{Gravitational allocation for roots of a Gaussian polynomial}\label{sec:gaussian-polynomial-roots}

In this section we study gravitational allocation to the roots of a
certain Gaussian random polynomial and prove Proposition \ref{prop:gaf}. Recall that we look at the polynomial given by \eqref{eq45}.
We bring the roots $\lambda_1, \ldots , \lambda_n\in\C$ of $p$ to the
sphere via stereographic projection. More explicitly, letting $P_n$ be
the rescaled stereographic projection map defined in Section
\ref{sec:prelim} and viewing the $\lambda_k$ as lying in the
horizontal plane in $\R^3$, it turns out that
\[ \L = \left\{ P_n^{-1}(\sqrt{n} \lambda_k) \right\}_{k=1}^n \]
is a rotationally equivariant random set of $n$ points on
$\S^2_n$. The rotational equivariance comes from the particular choice
of coefficients for $p$, see \cite[Chapter 2.3]{hkpv09}.

\begin{proof}[Proof of Proposition \ref{prop:gaf}]
  By \eqref{eq:L^1-identity} and rotational symmetry it suffices to
  compute $\E|F(x)|$ for any fixed point $x \in \S^2_n$. It is
  convenient to pick $x = (0, 0, -\sqrt{n/4\pi}) =
  P_n^{-1}(0)$. Letting $f$ be as in \eqref{eq:planar-F-def}, we then
  have
  \[ f(0, \sqrt{n}\lambda_k) = \frac{1}{\sqrt{n} \cdot \bar{\lambda}_k}, \]
  where complex numbers are interpreted as two-dimensional vectors in
  the horizontal plane. Using Proposition \ref{prop:P_n-properties} to
  convert between $F(x)$ and $f(0)$, we then have
  \[ F(x) = \sqrt{\pi} f(0) = \sqrt{\frac{\pi}{n}} \sum_{k = 1}^n \bar{\lambda}_k^{-1} = \sqrt{\frac{\pi}{n}} \cdot \frac{\bar{\zeta}_1 \cdot \sqrt{n}}{\bar{\zeta}_0 \cdot 1} = \sqrt{\pi} \cdot \frac{\bar{\zeta}_1}{\bar{\zeta}_0}, \]
  which gives a simple expression for $F$ in terms of two independent
  complex Gaussians. Taking expectations of the magnitude, we obtain
  \[ \E |F(x)| = \sqrt{\pi} \E \frac{|\bar{\zeta}_1|}{|\bar{\zeta}_0|} = \frac{\pi\sqrt{\pi}}{2}, \]
  which together with \eqref{eq:L^1-identity} establishes
  \eqref{eq:poly-roots-bound}.
\end{proof}

\section{Open problems}
\label{sec:openproblems}

\begin{enumerate}
     \item We have proved $O(\sqrt{\log n})$ bounds on typical
       distances for gravitational allocation to uniform points, but
       our results do not rule out the possibility of a small set of
       points with allocation distances much larger than $\sqrt{\log
         n}$ or, equivalently, of some allocation cells having large
       diameter. Let $z \in \L$ be chosen uniformly at random, and
       consider the cell $B(z)$ allocated to $z$. What is the law of
       the diameter of $B(z)$? Furthermore, what is the law of the
       maximal basin diameter, i.e., the law of $\max\{
       |x-\psi(x)|\,:\, x\in\BB S^2,\,\psi(x)\in \L \}$?

	\item The matching algorithm we consider in Corollary \ref{prop2}
      considers the gravitational field defined by the points
      $\cB$. One could attempt to define and analyze a matching
      algorithm where $\cA$ and $\cB$ are viewed as sets of particles
      undergoing dynamics where they exert attractive forces on
      particles of the opposite kind (as a variant, they may also
      repel particles of the same kind). One difficulty is that after
      the dynamics have evolved for some time the points are no longer
      uniformly distributed.

	\item In Corollary \ref{prop2} we consider a matching algorithm
      defined in terms of gravitational allocation. An alternative
      greedy matching algorithm can be obtained by iteratively
      matching nearest pairs of points, i.e., we find
      $i,j\in\{1,\dots,n\}$ such that $|a_i-b_j|$ is minimized, we
      define $\varphi(a_i)=b_j$, and we repeat the procedure with
      $\mcl A\setminus\{a_i \}$ and $\mcl B\setminus\{b_j
      \}$. \cite[Theorem 6]{hpps09} suggests that an upper bound for
      the average matching distance is
      $\int_{0}^{\sqrt{n}}r^{-0.496\dots}\,dr = \Theta(n^{0.252\dots
      })$. Can this bound be improved?
\end{enumerate}

\bibliography{mybib}
\bibliographystyle{hmralphaabbrv}

\end{document}